\documentclass[12pt]{amsart}
\usepackage{graphicx}
\usepackage[usenames]{color}
\usepackage{amssymb,amsmath,amsthm, fullpage, xcolor, soul}
\usepackage[mathscr]{eucal}
\usepackage{hyperref}
\usepackage{comment}
\newtheorem{theorem}{Theorem}[section]

\newtheorem{lemma}[theorem]{Lemma}
\newtheorem{cor}[theorem]{Corollary}
\newtheorem{proposition}[theorem]{Proposition}
\theoremstyle{definition}
\newtheorem{defn}[theorem]{Definition}

\newtheorem{remark}[theorem]{Remark}
\newtheorem*{remarks}{\bf Remarks}
\newtheorem*{theoremno}{Theorem}

\numberwithin{theorem}{section} 
\numberwithin{equation}{section}

\newcommand{\Z}{\mathbb Z}
\newcommand{\N}{\mathbb N}
\newcommand{\R}{\mathbb R}
\newcommand{\C}{\mathbb C}
\newcommand{\Q}{\mathbb Q}

\newcommand{\sm}[4]{\left(\begin{smallmatrix}#1&#2\\ #3&#4 \end{smallmatrix} \right)}

\newcommand{\qs}{Q_{\boldsymbol{\zeta}_{\boldsymbol{n,N}}}}
\newcommand{\qSubgroup}{\Gamma_{\boldsymbol{\zeta}_{\boldsymbol{n,N}}}}
\newcommand{\lcm}{\textnormal{lcm}}
\newcommand{\bs}{\boldsymbol}

\newcommand{\rmaf}{\color{black}}

\allowdisplaybreaks

\begin{document}
%%%%%%%%%%%%%%%%

\title{Quantum modular forms and singular combinatorial series with repeated roots of unity}
\author{Amanda Folsom} 
\address{Department of Mathematics and Statistics, Amherst College, Amherst, MA 01002}
\email{afolsom@amherst.edu}
\author{Min-Joo Jang}
\address{Department of Mathematics, Room 318, Run Run Shaw Building, 
The University of Hong Kong, 
Pokfulam, Hong Kong}
\email{min-joo.jang@hku.hk}
\author{Sam Kimport}
\address{Department of Mathematics, 450 Serra Mall,
Building 380,
Stanford, CA 94305-2125}
\email{skimport@stanford.edu}
\author{Holly Swisher}
\address{Department of Mathematics, Kidder Hall 368, Oregon State University, Corvallis, OR 97331-4605}
\email{swisherh@math.oregonstate.edu}

\thanks{{\bf{Acknowledgements.}}  Some of the ideas in this paper emerged at Banff International Research Station at the Women in Numbers 4 Workshop, and the authors are thankful for the hospitality and opportunity.  The first author is also grateful for the support of National Science Foundation Grant DMS-1449679, and the Simons Fellows in Mathematics Program.}



\maketitle

\begin{abstract}  In 2007, G.E. Andrews introduced the $(n+1)$-variable combinatorial generating function $R_n(x_1,x_2,\cdots,x_n;q)$  for ranks of $n$-marked Durfee symbols,  an $(n+1)$-dimensional multisum,  as a vast generalization to the ordinary two-variable partition rank generating function.  Since then, it has been a problem of interest to understand the automorphic properties of this function;  in special cases and under suitable specializations of parameters, $R_n$ has been shown to possess modular, quasimodular, and mock modular properties when viewed as a function on the upper half complex plane $\mathbb H$, in work of Bringmann, Folsom, Garvan, Kimport, Mahlburg, and Ono. Quantum modular forms, defined by Zagier in 2010, are similar to modular or mock modular forms but are defined on the rationals $\mathbb Q$ as opposed to $\mathbb H$, and exhibit modular transformations there up to suitably analytic error functions in $\mathbb R$;  in general, they have been related to diverse areas including number theory, topology, and representation theory.  Here,  we establish   quantum modular properties of $R_n$.    
\end{abstract}

%%%%%%%%%%%%%%%%%%%%%%%%%%%%%%%%%%%
%%%%%%%%%%%%%%%%%%%%%%%%%%%%%%%%%%%
%%%%%				    Introduction				   %%%%%
%%%%%%%%%%%%%%%%%%%%%%%%%%%%%%%%%%%
%%%%%%%%%%%%%%%%%%%%%%%%%%%%%%%%%%%
\section{Introduction and Statement of results}\label{intro}

\subsection{Background} 

Let $p(n)$ denote the number of partitions of a positive integer   $n$, where a \emph{partition} of $n$ is a non-increasing sequence of positive integers whose sum is $n$.  As an example, we see   there are 5 partitions of 4: $4,\ 3+1,\ 2+2,\ 2+1+1,\ 1+1+1+1$, and therefore $p(4)=5$. The generating function of $p(n)$ is given by 
\[
1+\sum_{n=1}^\infty p(n) q^n=\prod_{k=1}^\infty \frac{1}{1-q^k}=\frac{q^{\frac{1}{24}}}{\eta(\tau)},
\]
where
\begin{align}\label{def_eta} \eta(\tau) := q^{\frac{1}{24}}\prod_{k=1}^\infty (1-q^k) \end{align}
is {\it   Dedekind's $\eta$-function},    a weight $1/2$ modular form.   Here and throughout this section  we are setting $q=e^{2\pi i \tau}$, where $\tau \in \mathbb H:= \{x + i y \ | \ x \in \mathbb R,\  y \in \mathbb R^+\}$, the upper half of the complex plane.  

In order to provide a combinatorial proof of Ramanujan's remarkable partition congruences, Dyson \cite{Dyson} defined the \emph{rank} of a partition  as  the largest part of the partition minus the number of parts. He also defined the \emph{partition rank function} $N(m,n)$ to be the number of partitions of $n$ with rank equal to $m$.  If we set $N(m,0):=\delta_{m0}$ with $\delta_{ij}$ the Kronecker delta, and define the $q$-Pochhammer symbol for $n\in\N_0\cup\{\infty\}$ by $(a)_n=(a;q)_n:=\prod_{j=1}^n (1-aq^{j-1})$, then the generating function for $N(m,n)$ is given  by 
\begin{align}\label{rankgenfn}  \sum_{m=-\infty}^\infty \sum_{n=0}^\infty N(m,n) w^m q^n =   \sum_{n=0}^\infty \frac{q^{n^2}}{(wq;q)_n(w^{-1}q;q)_n} =: R_1(w;q).\end{align} 

Due to the deep connection between the rank generating function and the theory of modular forms, there have been many studies on the $q$-hypergeometric series defined in \eqref{rankgenfn}. For example, when $w=1$, one recovers the partition generating function, namely
\begin{equation}\label{r1mock1} 
R_1(1;q) =  \sum_{n=0}^\infty \frac{q^{n^2}}{(q;q)^2_n}=1+\sum_{n=1}^\infty p(n) q^n=\frac{ q^{\frac{1}{24}}}{\eta(\tau)},
\end{equation}
(essentially\footnote{Here and throughout, as is standard in this subject for simplicity's sake, we may slightly abuse terminology and refer to a function as a modular form or other modular object when in reality it must first be multiplied by a suitable power of $q$ to transform appropriately. }) the reciprocal of the Dedekind $\eta$-function, the modular form of weight $1/2$   defined in \eqref{def_eta}.
When $w=-1$, we have
\begin{equation}\label{r1mock2} 
R_1(-1;q) =   \sum_{n=0}^\infty \frac{q^{n^2}}{(-q;q)_n^2} =: f(q),
\end{equation}
where $f(q)$ is one of Ramanujan's third order mock theta functions \cite{BFOR}.  

Mock theta functions, and more generally mock modular forms and harmonic Maass forms have played central roles in modern number theory. In particular, for several decades after Ramanujan's death in 1920, no one understood how Ramanujan's mock theta functions fit into the theory of modular forms until the groundbreaking 2002 thesis of Zwegers \cite{Zwegers1}: we now know that Ramanujan's mock theta functions, a finite list of curious $q$-hypergeometric functions including $f(q)$,  are examples of \emph{mock modular forms}, the holomorphic parts of \emph{harmonic Maass forms}. In other words, they  exhibit suitable modular transformation properties after they are \emph{completed} by the addition of certain non-holomorphic functions.  Briefly speaking, harmonic Maass forms,   first defined by Bruiner and Funke \cite{BF}, are non-holomorphic generalizations of ordinary modular forms that in addition to satisfying appropriate modular transformations, must be eigenfunctions of a certain weight $k$-Laplacian operator, and satisfy suitable growth conditions in cusps (see \cite{BFOR, BF, OnoCDM, ZagierB} for more).   
 
  Motivated by  the fact  that specializing $R_1$ at $w=\pm 1$ yields two different modular objects, namely an ordinary modular form and a mock modular form as described  in \eqref{r1mock1} and \eqref{r1mock2},  
 Bringmann and Ono \cite{BO}  more generally proved  that upon specialization of the parameter $w$ to  complex roots of unity not equal to $1$, the rank generating function $R_1$  is also a mock modular form.  (See also \cite{ZagierB} for related work.) 
 
 \begin{theoremno}[\cite{BO} Theorem 1.1]  If $0<a<c$, then $$q^{-\frac{\ell_c}{24}}R_1(\zeta_c^a;q^{\ell_c}) + \frac{i \sin\left(\frac{\pi a}{c}\right) \ell_c^{\frac{1}{2}}}{\sqrt{3}} \int_{-\overline{\tau}}^{i\infty} \frac{\Theta\left(\frac{a}{c};\ell_c \rho\right)}{\sqrt{-i(\tau + \rho)}} d\rho $$ is a harmonic Maass form of weight $\frac{1}{2}$ on $\Gamma_c$.  
 \end{theoremno}
\noindent Here, $\zeta_c^a := e^{\frac{2\pi ia}{c}}$ is a $c$-th root of unity, $\Theta\left(\frac{a}{c};\ell_c\tau\right)$ is a certain weight $3/2$ cusp form, $\ell_c:=\lcm(2c^2,24)$, and $\Gamma_c$ is a particular subgroup of $\textnormal{SL}_2(\mathbb Z)$. 
   
In this paper we investigate modularity properties for a related combinatorial $q$-hypergeometric series, namely {\it the rank generating function for $n$-marked Durfee symbols}, as defined by Andrews in \cite{Andrews}.  Our results here extend prior work of the authors on this topic \cite{F-K, WIN}. 

We will not give details of the combinatoric objects called $n$-marked Durfee symbols themselves here,  and instead refer the reader to \cite{Andrews} for a full treatment, or \cite{WIN} for a brief overview.  However, we will note that the $n$-marked Durfee symbols are generalizations, using $n$ copies of the integers, of simpler objects called \emph{Durfee symbols}. Durfee symbols represent a partition's Ferrers diagram by indicating the size of the Durfee square, as well as the columns to the right and below the Durfee square.  For example, the Durfee symbol $$ \begin{pmatrix}2&\ \\ 2&1 \end{pmatrix}_3 $$ represents the partition. $4+4+3+2+1$ of $14$.  Andrews  defined   the \emph{rank} of a Durfee symbol to be the number of parts in the top row minus the number in the bottom row, which recovers Dyson's rank of the associated partition when $n=1$.  For the more general $n$-marked Durfee symbols, Andrews defines a notion of rank for each of the $n$ copies of the integers used.

Let $\mathcal{D}_n(m_1,m_2,\dots, m_n;r)$ denote the number of $n$-marked Durfee symbols arising from partitions of $r$ with $j$-th rank equal to $m_j$.  In \cite{Andrews},  Andrews showed that the $( n+1)$-variable generating function for Durfee symbols may be expressed in terms of certain $q$-hypergeometric series, analogous to (\ref{rankgenfn}).  To describe this, for $n\geq 2$, define 
{\small{\begin{align*}%\label{rkorigdef}
	&R_n({\boldsymbol{x}};q) := \\ & \nonumber \mathop{\sum_{m_1 > 0}}_{m_2,\dots,m_n \geq 0} \!\!\!\!\!\!\!\! \frac{q^{(m_1 + m_2 + \dots + m_n)^2 + (m_1 + \dots + m_{n-1}) + (m_1 + \dots + m_{n-2}) + \dots + m_1}}{(x_1q;q)_{m_1} \!\left(\frac{q}{x_1};q\right)_{m_1} \!\!\!\!(x_2 q^{m_1};q)_{m_2 + 1} \!\!\left(\frac{q^{m_1}}{x_2};q\right)_{m_2+1} \!\!\!\!\!\!\!\!\!\!\cdots(x_n q^{m_1 + \dots + m_{n-1}};q)_{m_n+1} \!\!\left(\!\frac{q^{m_1 + \dots + m_{n-1}}}{x_n};q\!\right)_{\! m_n+1}},\end{align*}}}\noindent where ${\boldsymbol{x}} = {\boldsymbol{x}}_n := (x_1,x_2,\dots,x_n).$ For $n=1$, the function $R_1(x;q)$ is defined as the $q$-hypergeometric series in (\ref{rankgenfn}).  
In what follows, for ease of notation, we may also write $R_1({\boldsymbol{x}};q)$ to denote $R_1(x;q)$, with the understanding that ${\boldsymbol{x}} := x$.  
In \cite{Andrews}, Andrews established the following result, generalizing (\ref{rankgenfn}).  
\begin{theoremno}[\cite{Andrews}  Theorem 10]  For $n\geq 1$ we have that \begin{align}\label{durfgenand1}\sum_{m_1,m_2,\dots,m_n = -\infty}^\infty \sum_{r=0}^\infty \mathcal{D}_n(m_1,m_2,\dots,m_n;r)x_1^{m_1}x_2^{m_2}\cdots x_n^{m_n}q^r = R_n({\boldsymbol{x}};q).\end{align}
\end{theoremno} 
When $n=1$, one recovers Dyson's rank,   in that  $\mathcal D_1(m_1;r)=N(m_1,r)$, so  we see  that \eqref{durfgenand1} reduces to \eqref{rankgenfn} in this case.  The mock modularity of the associated two variable generating function $R_1(x ;q)$ was established in \cite{BO} as described in the theorem above.   In \cite{Bri1}, Bringmann showed that $R_2(1,1;q)$ 
is a \emph{quasimock theta function},  and a year later
Bringmann, Garvan, and Mahlburg \cite{BGM} more generally proved that $R_{n}(1,1,\dots,1;q)$ is a quasimock theta function for $n\geq 2$.   Precise statements of these results can be found in \cite{Bri1, BGM}. 

 Two of the authors \cite{F-K}  established the automorphic properties of $R_n\left({\boldsymbol{x}};q\right)$ for more arbitrary parameters ${\boldsymbol{x}} = (x_1,x_2,\dots,x_n)$, thus  treating families of the rank generating functions for $n$-marked Durfee symbols  with additional singularities than those of $R_n(1,1,\dots,1;q)$.   The techniques of Andrews \cite{Andrews} and Bringmann \cite{Bri1} were not directly applicable in this instance  due to the presence of such additional singularities.   These singular combinatorial families are essentially mixed mock and quasimock modular forms.  Using this result, the authors \cite{WIN} established quantum modular properties of $R_n\left({\boldsymbol{x}};q\right)$ with distinct roots of unity $x_1, x_2, \dots, x_n$ as stated in the Theorem in Section 1.3 below. (See \cite{WIN} for more details.)  To precisely state the result from \cite{F-K}, we first introduce some notation, which we also use for the remainder of this paper. 
Namely, we consider functions evaluated at certain  length $n$ vectors ${\boldsymbol{\zeta}_{\boldsymbol{n,N}}} $  of roots of unity defined as follows (as in \cite{F-K}).   

Let $n$ and $N$ be fixed integers satisfying $0\leq N\leq \big\lfloor\frac{n}{2}\big\rfloor$, and $n\geq 2$.  Suppose for $1\leq j \leq n-N$, $\alpha_j \in \mathbb Z$ and $\beta_j \in  \mathbb N$, where $\beta_j \nmid \alpha_j, \beta_j \nmid 2\alpha_j$, and that $\frac{\alpha_{r}}{\beta_{r}} \pm \frac{\alpha_{s}}{\beta_{s}} \not\in\mathbb Z$ if $1\leq r\neq s \leq n-N$. Let
\begin{align}%\label{alphavecbold}
	{\boldsymbol{\alpha}_{\boldsymbol{n,N}}}&:= \Big(\underbrace{\frac{\alpha_1}{\beta_1},\frac{\alpha_1}{\beta_1},\frac{\alpha_2}{\beta_2},\frac{\alpha_2}{\beta_2},\dots,\frac{\alpha_N}{\beta_N},\frac{\alpha_N}{\beta_N}}_{2N},\underbrace{\frac{\alpha_{N+1}}{\beta_{N+1}},\frac{\alpha_{N+2}}{\beta_{N+2}},\dots,\frac{\alpha_{n-N}}{\beta_{n-N}}}_{n-2N}\Big) \in \mathbb Q^n  \notag\\ 
	\label{zetavec}
	{\boldsymbol{\zeta}_{\boldsymbol{n,N}}} &:=\big(\underbrace{\zeta_{\beta_1}^{\alpha_1},\zeta_{\beta_1}^{\alpha_1},\zeta_{\beta_2}^{\alpha_2},\zeta_{\beta_2}^{\alpha_2},\dots,\zeta_{\beta_N}^{\alpha_N},\zeta_{\beta_N}^{\alpha_N}}_{2N},\underbrace{\zeta_{\beta_{N+1}}^{\alpha_{N+1}},\zeta_{\beta_{N+2}}^{\alpha_{N+2}},\dots,\zeta_{\beta_{n-N}}^{\alpha_{n-N}}}_{n-2N}\big) \in \mathbb C^n.
\end{align}
Here, $\zeta_{\beta}^\alpha = e^{2\pi i \frac{\alpha}{\beta}}$ as before.  
\begin{remark} We point out that the dependence of the vector $\boldsymbol{\zeta_{n,N}}$ on $n$ is reflected only in the length of the vector, and not (necessarily) in the roots of unity that comprise its components.  In particular, the vector components may be chosen to be $m$-th roots of unity for different values of $m$.   
 \end{remark}

\begin{remark}  The conditions  given in \cite{F-K}  do not require that $\gcd(\alpha_j, \beta_j) = 1$.  Instead, they merely require  that $\frac{\alpha_j}{\beta_j} \neq \frac{1}{2}\Z$. Without loss of generality, we will assume here  that $\gcd(\alpha_j, \beta_j) = 1$.  Then, requiring that $\beta_j \nmid 2\alpha_j$ is the same as saying $\beta_j \neq 2$. 
\end{remark}
  
 To complete the function $R_n(\boldsymbol{\zeta}_{\boldsymbol{n,N}};q)$ 
we first define the holomorphic function 
\begin{align*} B_n^+(\boldsymbol{\zeta}_{\boldsymbol{n,N}};q) := R_n(\boldsymbol{\zeta}_{\boldsymbol{n,N}};q) + 
b_n(\boldsymbol{\zeta}_{\boldsymbol{n,N}};q), 
\end{align*} 
where $b_n(\boldsymbol{\zeta}_{\boldsymbol{n,N}}; q)$ is  defined by 
\begin{align*}%\label{bKdef}
	b_n(\boldsymbol{\zeta}_{\boldsymbol{n,N}}; q) := &\, \frac{1}{(q)_\infty} \sum_{j=1}^N \zeta_{2\beta_j}^{-\alpha_j} \frac{\zeta_{\beta_j}^{-\alpha_j}}{2}\left(\frac{3}{\Pi_j(\boldsymbol{\alpha_{\boldsymbol{n,N}}}, 0)} + \frac{\left.\frac{d}{dw}\Pi_j(\boldsymbol{\alpha_{\boldsymbol{n,N}}}, w)\right\vert_{w=0}}{\pi i (\Pi_j(\boldsymbol{\alpha_{\boldsymbol{n,N}}}, 0))^2}\right) A_3\left(\frac{\alpha_j}{\beta_j}, -2\tau; \tau\right)\\
	&- \frac{1}{(q)_\infty} \sum_{j=1}^N \zeta_{2\beta_j}^{-3\alpha_j} \frac{\zeta_{\beta_j}^{-\alpha_j}}{2} \left(\frac{1}{\Pi_j(\boldsymbol{\alpha_{\boldsymbol{n,N}}}, 0)} + \frac{\left.\frac{d}{dw}\Pi_j(\boldsymbol{\alpha_{\boldsymbol{n,N}}}, w)\right\vert_{w=0}}{\pi i (\Pi_j(\boldsymbol{\alpha_{\boldsymbol{n,N}}}, 0))^2}\right) A_3\left(\frac{\alpha_j}{\beta_j}, -2\tau; \tau\right).
\end{align*}
Here, $\Pi_j$ is a constant depending only on $\boldsymbol{\zeta}_{\boldsymbol{n,N}}$ as defined in \cite{F-K}, and $A_3$ is the level three Appell function (see \cite{BFOR} or \cite{Zwegers2})  
\begin{equation}\label{A3def}A_3(u, v; \tau) := e^{3\pi i u} \sum_{n\in\Z} \frac{(-1)^n q^{3n(n+1)/2}e^{2\pi i v}}{1 - e^{2\pi i u}q^n},\end{equation}
where $u, v\in\C$. 
In \cite{Zwegers2}, Zwegers showed that $A_3(u, v;\tau)$ can be completed using the non-holomorphic function $\mathscr R_3$ in (\ref{AminusDef}) to transform like a non-holomorphic Jacobi form.  Using these functions, as in \cite{F-K} we let
\begin{align*}%\label{def_Bnhatdef}
\widehat{B}_n(\boldsymbol{\zeta}_{\boldsymbol{n,N}};q):= q^{-\frac{1}{24}}(B_n^+(\boldsymbol{\zeta}_{\boldsymbol{n,N}};q)+B_n^-(\boldsymbol{\zeta}_{\boldsymbol{n,N}};q)),  
\end{align*}
where the function $B_n^-$ is given explicitly in terms of sums of functions involving $F_{m,3}^-$ (see \eqref{def_Fmsminus}) and $\mathscr R_3$ (see (\ref{AminusDef})) in \cite[equation (4.3)]{F-K}.  We have the following theorem, established by two of the authors in \cite{F-K}.
 \begin{theoremno}[\cite{F-K} Theorem 1.1] If $n\geq 2$ is an integer, and $N$ is an integer satisfying $0\leq N \leq \left\lfloor \frac{n}{2}\right\rfloor$,  then  $\widehat{B}_n\!\left( {\boldsymbol{\zeta}_{\boldsymbol{n,N}}};q \right) = \widehat{\mathcal{H}}\!\left( {\boldsymbol{\zeta}_{\boldsymbol{n,N}}};q \right) + \widehat{\mathcal A}\!\left( {\boldsymbol{\zeta}_{\boldsymbol{n,N}}};q \right)$, where $\widehat{\mathcal{H}}\!\left( {\boldsymbol{\zeta}_{\boldsymbol{n,N}}};q \right)$  and $\widehat{\mathcal{A}}\!\left( {\boldsymbol{\zeta}_{\boldsymbol{n,N}}};q \right)$ are non-holomorphic modular forms of weights  $3/2$ and $1/2$, respectively, on $\Gamma_{n,N}$, with character $\chi_\gamma^{-1}$. 
\end{theoremno}  \noindent Here, the functions
  $\widehat{\mathcal{H}}\!\left( {\boldsymbol{\zeta}_{\boldsymbol{n,N}}};q \right)$ and $\widehat{\mathcal{A}}\!\left( {\boldsymbol{\zeta}_{\boldsymbol{n,N}}};q \right)$,  as well as their holomorphic parts $\mathcal{H}\!\left( {\boldsymbol{\zeta}_{\boldsymbol{n,N}}};q \right)$ and $\mathcal{A}\!\left( {\boldsymbol{\zeta}_{\boldsymbol{n,N}}};q \right)$, are defined in (\ref{hatcalh}) and (\ref{hatcala}), respectively. The subgroup $\Gamma_{n,N}\subseteq \textnormal{SL}_2(\mathbb Z)$ under which $\widehat{B}_n({\boldsymbol{\zeta}_{\boldsymbol{n,N}}};q)$ transforms is defined by
	\[
	\Gamma_{n,N}:=\bigcap_{j=1}^{ n-N} \Gamma_0\left(2\beta_j^2\right)\cap \Gamma_1(2\beta_j),
	\] and the Nebentypus character $\chi_\gamma$ is given in Lemma \ref{ETtrans}.

\begin{remark}\label{zagremark} Zagier   defined a \emph{mixed mock modular form} \cite{BFOR, Zlec} to be the product of a mock modular form and a modular form.  Here, the holomorphic parts of $\widehat{B}_n$ consist of linear combinations of mixed mock modular forms, and also terms consisting of derivatives $\frac{d}{du} \phi(u,\tau) \big |_{u=0}$ of mock Jacobi forms $\phi(u,\tau)$ in the Jacobi $u$ variable evaluated at $u$=0, multiplied by modular forms.  For simplicity, we  may still refer to holomorphic parts of $\widehat{B}_n\!\left( {\boldsymbol{\zeta_{n,N}}};q \right)$    as \emph{mixed mock modular forms}.   
\end{remark}

\subsection{Quantum modular forms}  
In this paper, we  extend   results from \cite{WIN}, which establish quantum modular properties for the $(n+1)$-variable rank generating function for $n$-marked Durfee symbols $R_n({\boldsymbol{x}};q)$ with distinct roots of unity $x_1, x_2, \dots, x_n$, by  determining quantum modular properties for $R_n({\boldsymbol{x}};q)$ when there are repeated roots of unity. 

Loosely speaking, a quantum modular form is similar to a mock modular form in that it exhibits a modular-like transformation with respect to the action of a suitable subgroup of $\textnormal{SL}_2(\mathbb Z)$;  however, rather than the upper half-plane $\mathbb H$, the domain of a quantum modular form is the set of rationals $\mathbb Q$ or an appropriate subset.  
The formal definition of a quantum modular form was originally introduced by Zagier in \cite{Zqmf} and has  since   been slightly modified to allow for half-integral weights, subgroups of $\operatorname{SL_2}(\mathbb{Z})$, etc.\ (see \cite{BFOR}).

\begin{defn} \label{qmf}
A weight $k \in \frac{1}{2} \mathbb{Z}$ quantum modular form is a complex-valued function $f$ on $\mathbb{Q}$, such that for all $\gamma = \sm abcd \in \operatorname{SL_2}(\mathbb{Z})$, the functions $h_\gamma: \mathbb{Q} \setminus \gamma^{-1}(i\infty) \rightarrow \mathbb{C}$ defined by  
\begin{equation*}
h_\gamma(x) := f(x)-\varepsilon^{-1}(\gamma) (cx+d)^{-k} f\left(\frac{ax+b}{cx+d}\right)
\end{equation*}
satisfy a ``suitable" property of continuity or analyticity in a subset of $\mathbb{R}$.  
\end{defn}
\begin{remarks} \quad 
\begin{enumerate}
\item The complex numbers $\varepsilon(\gamma)$, which satisfy $|\varepsilon(\gamma)|=1$, are such as those appearing in the theory of half-integral weight modular forms. 
\item We may modify Definition \ref{qmf} appropriately to allow transformations on appropriate subgroups of $\operatorname{SL_2}(\mathbb{Z})$. We may also restrict the domains of the functions $h_\gamma$ to be suitable subsets of $\mathbb{Q}$. 
\end{enumerate}
\end{remarks}

 Since Zagier's initial definition, the subject of quantum modular forms has been widely studied   (see \cite{BFOR} and references therein for a number of examples and applications).  In particular, the notion of a quantum modular form is now known to have a direct connection to Ramanujan's original definition of a mock theta function \cite{BR, FOR}  and more generally to that of   a mock modular form \cite{CLR}.   
 
\subsection{Results} 
 Although automorphic properties of the rank generating function for $n$-marked Durfee symbols $R_n$ in (\ref{durfgenand1}) on $\mathbb H$ have been established by two of the authors (see  \cite[Theorem 1.1]{F-K} above) and $\mathbb Q$ is a natural boundary to $\mathbb H$, a priori there is no reason to expect $R_n$ to converge on $\mathbb Q$, let alone exhibit quantum-automorphic properties there.  However, here (as well as in previous work \cite{WIN}) we do in fact establish quantum-automorphic properties for $R_n$.   

For the remainder of this paper, we use the notation $$\mathcal V_{n,N}(\tau) := \mathcal V({\boldsymbol{\zeta}_{\boldsymbol{n,N}}};q),$$ where $\mathcal V$ may refer to any one of the functions  
 $$\widehat{\mathcal A}, \mathcal A, \widehat{\mathcal H}, \mathcal H, \widehat{B}_n , R_n, b_n, B_n^+, B_n^-. $$ 
  (We  omit repetitive subscripts and write $\mathcal V_{n,N}(\tau)$ for $(\mathcal V_n)_{n,N}(\tau)$ as well.)    Note that when $N=0$ these functions are equal to the ones in \cite{WIN}. Namely, $\mathcal V_{n,0}(\tau)=\mathcal V_n(\tau)$.    
  
In \cite{WIN}, we established the quantum modular properties of $R_n$ in the  special case when $N=0$. More precisely,   we  showed that for $N=0$, $\mathcal A_{n,N}(\tau)   =  q^{-\frac{1}{24}}R_n(\boldsymbol{\zeta}_{\boldsymbol{n,N}};q)$ is a quantum modular form 
under the  action of a subgroup of $\Gamma_{n,0},$ with quantum set 
\begin{equation}\label{qSetDef} \qs := \left\{\frac{h}{k}\in \Q\; \middle\vert\; \begin{aligned} & \ h\in\Z, k\in\N, \gcd(h,k) = 1, \ \beta_j \nmid k\ \forall\ 1\le j\le n,\\&\left\vert \frac{\alpha_j}{\beta_j}k - \left[\frac{\alpha_j}{\beta_j}k\right]\right\vert > \frac{1}{6}\ \forall\ 1\le j\le n\end{aligned} \right\},\end{equation}
where $[x]$ denotes  the closest integer to $x$.  

 \begin{remark} For $x \in \frac12 + \mathbb Z$,  different sources define $[x]$ to mean
either $x-\frac12$ or $x+\frac12$. The definition of $\qs$ involving $[ \cdot ]$ is well-defined for either of these conventions in the case of $x\in \frac12 + \mathbb Z,$ as $\vert x - [x]\vert = \frac{1}{2}$.\end{remark}

 Here, we consider the complementary case of $N>0,$ and ultimately establish  quantum modular properties  for the function $q^{-\frac{1}{24}}B_{n,N}^+$ in this setting.  When $N>0$, one has repeated roots of unity in (\ref{zetavec}).  This leads to additional singularities, rendering the study of the modular properties of $q^{-\frac{1}{24}}B_{n,N}^+$ in the case $N>0$ significantly more complex than in   the  case $N=0$.  Before stating our main result,   we first define  
\begin{equation}\label{def_ell}\ell = \ell(\bs{\zeta_{n,N}}):= \begin{cases} 6\left[\text{lcm}(\beta_1, \dots, \beta_{n})\right]^2 &\text{ if } 3 \nmid \beta_j \text{ for all } 1\leq j \leq n, \\ 2\left[\text{lcm}(\beta_1, \dots, \beta_{n})\right]^2 &\text{ if } 3 \mid \beta_j \text{ for some } 1\leq j \leq n,  \end{cases}\end{equation} and let $S_\ell:=\left(\begin{smallmatrix}1 & 0 \\ \ell & 1 \end{smallmatrix}\right)$, $T:=\left(\begin{smallmatrix}1 & 1 \\ 0 & 1 \end{smallmatrix}\right)$.  We define the group generated by these two matrices as 
\[ %\label{eqn:GroupDefn}
\qSubgroup:= \langle S_\ell, T \rangle.
\] Moreover,  the constant (a finite product) $\Pi_{j}^\dagger({\boldsymbol{\alpha}_{\boldsymbol{n,N}}})$ is defined explicitly in \cite[(4.2)]{F-K} (where one must replace $n\mapsto j$ and $k\mapsto n$).   Throughout the paper we let $e(x):=e^{2\pi ix}$.

\begin{theoremno}[\cite{WIN} Theorem 1.7]
Let $N=0$. For all $\gamma = \left(\begin{smallmatrix}a & b \\ c & d \end{smallmatrix}\right) \in \qSubgroup$, and $x\in \qs$, \[H_{n,\gamma}(x) := \mathcal{A}_n(x) - \chi_\gamma (c x+ d)^{-\frac12}\mathcal{A}_n(\gamma x) \] 
is defined, and extends to an analytic function in $x$ on $\mathbb{R} - \{\frac{-c}{d}\}$. 
In particular, for the matrix $S_\ell$, 
\begin{multline}\label{eqn_Hsell}
H_{n,S_\ell}(x) = \frac{\sqrt{3}}{2} \sum_{j=1}^{n}\frac{(\zeta_{2\beta_j}^{-3\alpha_j} - \zeta_{2\beta_j}^{-\alpha_j})}{\displaystyle\Pi^\dag_j( {\bs{\alpha_{n,N}}})} e \left(\frac{2\alpha_j}{\beta_j} \right) \left[\sum_\pm \zeta_6^{\pm1}\int_{\frac{1}{\ell}}^{i\infty}\frac{g_{\pm\frac13+\frac12,-\frac{3\alpha_j}{\beta_j}+\frac12}(3\rho)}{\sqrt{-i(\rho+x)}}d\rho \right] \\
+\sum_{j=1}^{n}\frac{(\zeta_{2\beta_j}^{-3\alpha_j} - \zeta_{2\beta_j}^{-\alpha_j})}{\displaystyle\Pi^\dag_j( {\bs{\alpha_{n,N}}})} (\ell x+1)^{-\frac12}\zeta_{24}^{-\ell}\mathcal{E}_1\left(\frac{\alpha_j}{\beta_j},\ell;x\right),
\end{multline}
where  the weight $3/2$ theta functions $g_{a,b}$ are defined in \eqref{def_gab}, and $\mathcal E_1$ is defined in Lemma \ref{lem_Stransform}.
\end{theoremno}

As described above,  for the case of $N > 0$, recall that there is an additional holomorphic function   $b_n(\boldsymbol{\zeta}_{\boldsymbol{n,N}}; q)$ which is added to   $R_n(\boldsymbol{\zeta}_{\boldsymbol{n,N}}; q)$  to  obtain a ``modular" object. (See \cite[Theorem 1.1]{F-K} above.)
 For $N\geq 0$, we have the following result which generalizes \cite[Theorem 1.7]{WIN} above. 

\begin{theorem}\label{thm_main_NN}      For any integer $N\geq 0$ we have that 
$$ e^{-\frac{\pi i x}{12}} B^+_{n,N}(x)   = \mathcal H_{n,N}(x)   + \mathcal A_{n,N}(x),$$  
where $\mathcal H_{n,N}$ is a quantum modular form of weight $3/2$, and $\mathcal A_{n,N}$ is a quantum modular form of weight $1/2$, both   defined on the quantum set $Q_{\boldsymbol{\zeta_{n,N}}}$ with respect to the group $\Gamma_{\boldsymbol{\zeta_{n,N}}}$ and with character  $\chi_\gamma^{-1}$.     
That is, for all $\gamma=\sm{a}{b}{c}{d} \in \Gamma_{\boldsymbol{\zeta_{n,N}}}$ and $x\in Q_{\boldsymbol{\zeta_{n,N}}}$, we have that 
\begin{align*}
H_{n,N,\gamma}^{(1)}(x) := \mathcal{A}_{n,N}(x) - \chi_\gamma (c x+ d)^{-\frac12}\mathcal{A}_{n,N}(\gamma x)
 \end{align*} and 
 \begin{align*}
 H_{n,N,\gamma}^{(2)}(x):= \mathcal H_{n,N}(x) -  \chi_\gamma (cx + d)^{-\frac32}\mathcal H_{n,N}(\gamma x). 
\end{align*}
are defined, and extend  to   analytic functions in $x$ on $\mathbb{R} - \{\frac{-c}{d}\}$. 

In particular, for the matrix $S_\ell$, 
$H_{n,N,S_\ell}^{(1)}(x) = H_{n,S_\ell}(x),$ where $H_{n,S_\ell}(x)$ is as in \eqref{eqn_Hsell}, and 
\begin{align*}
H_{n,N,S_\ell}^{(2)}(x) =& 
-\zeta_{24}^{-\ell}(\ell x + 1)^{-\frac32}\Bigg(\sum_{j=1}^N \frac{\zeta_{2\beta_j}^{\alpha_j} - \zeta_{2\beta_j}^{-\alpha_j}}{2\Pi_j(0)}  \\  
&\times\Bigg[ \sqrt{\ell x + 1}\zeta_{24}^\ell\left(\left(\frac{\ell}{2}-3\frac{\alpha_j}{\beta_j}\ell\right)H_{\alpha_j,\beta_j}(x) 
 -\frac{1}{2\pi i}D_{\alpha_j,\beta_j}(x) \right)
+\mathcal E_2\left(\frac{\alpha_j}{\beta_j},\ell;x\right)\Bigg]
\Bigg),
\end{align*}
  where $H_{\alpha,\beta}$ is as in \eqref{Hperiodint}, $D_{\alpha,\beta}$ is defined in  \eqref{def_Dab}, and $\mathcal E_2$ is defined in Proposition \ref{prop_hmintsf}.  \rmaf
\end{theorem}  
 \begin{remark}  Our results reveal that $e^{-\frac{\pi i x}{12}} B^+_{n,N}(x)$ is a  mixed weight  quantum modular form.  From this one also obtains the analytic nature of $H_{n,N,\gamma}^{(1)}(x)(cx+d)^{-1} + H_{n,N,\gamma}^{(2)}(x),$   which showcases the transformation of $R_{n,N}(x)$. \end{remark} 
 \begin{remark}  By combining the explicit closed-form evaluation of the function $R_n({\boldsymbol{\zeta_{n,N}}};\zeta_k^h)$ as a rational polynomial in roots of unity given in Section \ref{quantumSet} with the quantum modular transformations from Theorem \ref{thm_main_NN}, we obtain explicit evaluations of Eichler integrals of (derivatives of) modular forms.   Similar corollaries have been explicitly established in \cite{FGKST, FKTY}.
\end{remark}
 
%%%%%%%%%%%%%%%%%%%%%%%%%%%%%%%%%%%
%%%%%%%%%%%%%%%%%%%%%%%%%%%%%%%%%%%
%%%%%				   Preliminaries				   %%%%%
%%%%%%%%%%%%%%%%%%%%%%%%%%%%%%%%%%%
%%%%%%%%%%%%%%%%%%%%%%%%%%%%%%%%%%%
\section{Preliminaries}\label{prelim}  
\subsection{Modular, mock modular and Jacobi forms}
 The Dedekind $\eta$-function, defined in (\ref{def_eta}), is a well-known modular form of weight $1/2$.  It transforms with character
 $\chi_\gamma$ (see \cite[Ch. 4, Thm. 2]{Knopp}):
\[%\label{Chi_gammaForm}
\chi_\gamma =
\left\{ \begin{array}{ll}
\big(\frac{d}{|c|} \big)e\left(\frac{1}{24}\left( (a+d)c - bd(c^2-1) - 3c \right)\right)
& \mbox{ if } c \equiv 1 \pmod{2}, \\				
\big( \frac{c}{d} \big) e\left(\frac{1}{24}\left( (a+d)c - bd(c^2-1) + 3d - 3 - 3cd \right)\right)
& \mbox{ if } d\equiv 1\pmod{2},
\end{array}\right.
\]
where $\gamma  = \sm{a}{b}{c}{d} \in \textnormal{SL}_2(\mathbb Z)$, and $\big(\frac{\alpha}{\beta}\big)$ is the generalized Legendre symbol.
Precisely, $\eta$ satisfies the following transformation property \cite{Rad}.
\begin{lemma}\label{ETtrans}
 For
$\gamma=\sm{a}{b}{c}{d} \in \textnormal{SL}_2(\mathbb Z)$, we have that
$
\eta\left(\gamma\tau\right)  = \chi_\gamma(c\tau + d)^{\frac{1}{2}} \eta(\tau).
$
 \end{lemma}  
   
We require two additional ``modular" objects, namely the Jacobi theta function $\vartheta(u;\tau)$, an ordinary Jacobi form, and a non-holomorphic function $R(u;\tau)$ used by Zwegers in \cite{Zwegers1}.    In what follows, we will also need certain transformation properties of these functions.

\begin{proposition} \label{thetaTransform} For $u \in\C$ and $\tau\in\mathbb{H}$, define
\begin{equation}\label{thetaDef}\vartheta(u;\tau) := \sum_{\nu\in\frac{1}{2} + \Z} e^{\pi i \nu^2\tau + 2\pi i \nu\left(u + \frac{1}{2}\right)}.\end{equation}
Then $\vartheta$ satisfies
\begin{enumerate}
	\item $\vartheta(u+1; \tau) = -\vartheta(u; \tau),$\\
	\item $\vartheta(u + \tau; \tau) = -e^{-\pi i \tau - 2\pi i u}\vartheta(u; \tau),$\\
	\item $\displaystyle \vartheta(u; \tau) = - i e^{\pi i \tau/4}e^{-\pi i u} \prod_{m=1}^\infty (1-e^{2\pi i m\tau})(1-e^{2\pi i u}e^{2\pi i \tau(m-1)})(1 - e^{-2\pi i u}e^{2\pi i m\tau}).$
\end{enumerate}
\end{proposition}  
The non-holomorphic function $R(u;\tau)$   is defined in \cite{Zwegers1} by
\[
R(u;\tau):=\sum_{\nu\in\frac12+\Z} \left\{\operatorname{sgn}(\nu)-E\left(\left(\nu+\frac{\operatorname{Im}(u)}{\operatorname{Im}(\tau)}\right)\sqrt{2\operatorname{Im}(\tau)}\right)\right\}(-1)^{\nu-\frac12}e^{-\pi i\nu^2\tau-2\pi i\nu u},
\]
where
\[
E(z):=2\int_0^ze^{-\pi t^2}dt.
\] The function $R$ transforms like a (non-holomorphic) mock Jacobi form as follows.  
\begin{proposition}[Propositions 1.9 and 1.10, \cite{Zwegers1}]\label{Rtransform} The function $R$ satsifies the following   transformation properties:
\begin{enumerate}
	\item $R(u+1;\tau) = -R(u; \tau),$\\
	\item $R(u; \tau) + e^{-2\pi i u - \pi i\tau}R(u+\tau; \tau) = 2e^{-\pi i u - \pi i \tau/4}$,\\
	\item  $R(u;\tau) = R(-u;\tau)$,   \\ 
        \item $R(u;\tau+1)=e^{-\frac{\pi i}{4}} R(u;\tau)$,\\
        \item $\frac{1}{\sqrt{-i\tau}} e^{\pi i u^2/\tau} R\left(\frac{u}{\tau};-\frac{1}{\tau}\right)+R(u;\tau)=h(u;\tau),$
where the Mordell integral is defined by \begin{align}\label{def_hmordell} h(u;\tau):=\int_\R \frac{e^{\pi i\tau t^2-2\pi ut}}{\cosh \pi t} dt. \end{align} 
 
\end{enumerate}
\end{proposition}

 Using the functions $\vartheta$ and $R$, Zwegers  defined the completion of $A_3(u, v; \tau)$ (see \eqref{A3def})  by 
\begin{align}\label{def_A3hat} \widehat{A}_3(u, v; \tau) := A_3(u, v;\tau) + \mathscr R_3(u, v;\tau),\end{align}
where
\begin{align}\label{AminusDef}
	\mathscr R_3 (u, v;\tau) :=& \frac{i}{2} \sum_{j=0}^{2} e^{2\pi i j u} \vartheta(v + j\tau + 1; 3\tau) R(3u - v - j\tau - 1; 3\tau)\\
	=& \frac{i}{2} \sum_{j=0}^{2} e^{2\pi i j u} \vartheta(v + j\tau; 3\tau) R(3u - v - j\tau; 3\tau),\nonumber
\end{align}
where the equality of the two expressions in \eqref{AminusDef} is justified by Proposition \ref{thetaTransform} and Proposition \ref{Rtransform}.   This completed  function transforms like a (non-holmorphic) Jacobi form, and in particular satisfies  the following elliptic transformation.  

\begin{theorem}[{\cite[Theorem 2.2]{Zwegers2}}]\label{completeAtransform} For $n_1, n_2, m_1, m_2\in\Z$, the completed level $3$ Appell function $\widehat{A}_3$ satisfies
\[\widehat{A}_3(u + n_1\tau + m_1, v + n_2\tau + m_2; \tau) = (-1)^{n_1 + m_1}e^{2\pi i (u(3n_1 - n_2) - vn_1)}q^{3n_1^2/2 - n_1n_2}\widehat{A}_3(u, v; \tau).\]
\end{theorem}

 We will also make use of the following results on the Mordell integral defined in \eqref{def_hmordell}.
 
 \begin{theorem}[{\cite[Theorem 1.2 (1), (2), (4)]{Zwegers1}}]\label{thm_Zh}  Let $z\in \mathbb C, \tau \in \mathbb H$. We have that  \begin{enumerate}
 \item $h(z;\tau) + h(z+1;\tau) = \frac{2}{\sqrt{-i\tau}} e^{\pi i \frac{(z+\frac12)^2}{\tau}},$ \\
 \item $h(z;\tau) + e\left(-z-\frac{\tau}{2}\right)h(z+\tau;\tau)= 2 e\left(-\frac{z}{2} - \frac{\tau}{8}\right), $\\ 
 \item $h$ is an even function of $z$.
 \end{enumerate}
 \end{theorem}  
 
 Zwegers also showed how under certain hypotheses, the functions $h$ and $R$ can be written in terms of integrals involving the weight $3/2$ modular forms $g_{a,b}(\tau)$, defined for $a,b\in\mathbb R$ and $\tau \in \mathbb H$ by
 \begin{align}\label{def_gab}
 g_{a,b}(\tau) := \sum_{\nu \in a + \mathbb Z} \nu e^{\pi i \nu^2\tau + 2\pi i \nu b}.
 \end{align}

We have the following properties of $g_{a,b}$.
\begin{proposition}[{\cite[Proposition 1.15 (1), (2), (4), (5)]{Zwegers1}}]\label{prop_Zg}
The function $g_{a,b}$ satisfies:
\begin{enumerate}
\item[(1)] $g_{a+1,b}(\tau)= g_{a,b}(\tau)$, \\
\item[(2)] $g_{a,b+1}(\tau)= e^{2\pi ia} g_{a,b}(\tau)$,\\
\item[(3)] $g_{a,b}(\tau+1)= e^{-\pi ia(a+1)}g_{a,a+b+\frac12}(\tau)$,\\
\item[(4)] $g_{a,b}(-\frac{1}{\tau})= i e^{2\pi iab} (-i\tau)^{\frac32}g_{b,-a}(\tau)$.
\end{enumerate}
\end{proposition}
 
  \begin{theorem}[{\cite[Theorem 1.16 (2)]{Zwegers1}}]\label{thm_Zh2}
 Let $\tau \in \mathbb H$. For $a,b \in (-\frac12,\frac12)$, we have
 $$h(a\tau-b;\tau) = -  e\left(\tfrac{a^2\tau}{2} - a(b+\tfrac12)\right) \int_{0}^{i\infty} \frac{g_{a+\frac12,b+\frac12}(\rho)}{\sqrt{-i(\rho+\tau)}}d\rho.$$ 
 \end{theorem}
 
\subsection{Completing the function $R_{n,N}$}\label{subsec_compl} 
 Here we review some preliminary results and functions from \cite{F-K}.   Recall that
 $n$ and $N$ are fixed integers satisfying $0\leq N\leq \big\lfloor\frac{n}{2}\big\rfloor$, and $n\geq 2$.  The vectors 
 ${\boldsymbol{\zeta}_{\boldsymbol{n,N}}}$ are defined in \eqref{zetavec}.    To complete the function $R_{n,N}(\tau) := R(\boldsymbol{\zeta}_{\boldsymbol{n,N}};q)$ as described in $\S \ref{intro}$, we use the functions from \cite{F-K} 
 \begin{align} \nonumber
{F}^+_{m,s}(\boldsymbol{x};\tau) \!\!&:=\!  \lim_{w\to 0} \!\left(\!\frac{e\left(- x_m \right)}{e(w) - e(-w)} \!\!\left(\!\!e^{s\pi i w}\frac{A_3\left(-w +  x_m ,-2\tau;\tau\!\right)}{\Pi_m( \boldsymbol{x},-w)}\!
 -\!e^{-s \pi i w} \frac{A_3 \left(w +  x_m,-2\tau;\tau \right)}{\Pi_m(\boldsymbol{x},w)}\!\right) \!\!\right)\!\!,  \\ \nonumber &  \\\label{def_Fmsminus} 
{F}^-_{m,s}(\boldsymbol{x};\tau) &\!\!:=\!  \lim_{w\to 0} \!\left(\!\frac{e\left(- x_m \right)}{e(w) - e(-w)} \!\!\left(\!\!e^{s\pi i w}\frac{\mathscr R_3\left(-w +  x_m,-2\tau;\tau\!\right)}{\Pi_m( \boldsymbol{x},-w)}\!
 -\!e^{-s \pi i w} \frac{\mathscr R_3 \left(w +  x_m ,-2\tau;\tau \right)}{\Pi_m( \boldsymbol{x},w)}\right) \!\!\right)\!\!, \end{align}
 with ${A}_3$ as defined in \eqref{A3def},  $\mathscr R_3$ as defined in \eqref{AminusDef}, and $\Pi_m({ \boldsymbol x},w)$ as defined explicitly in \cite{F-K}. 
The corresponding completed function is: 
\begin{align*}%\label{fhatgen}
\widehat{F}_{m,s}(\boldsymbol{x};\tau) &:= F^+_{m,s}(\boldsymbol{x};\tau) + F^-_{m,s}(\boldsymbol{x};\tau).\end{align*}  

Using $\widehat{A}_3$ (see \eqref{def_A3hat}), we also define \begin{align*}
  \widehat{G}_{m,s}( {\boldsymbol{\alpha}_{\boldsymbol{n,N}}};\tau) &:= \frac{\zeta_{\beta_m}^{-\alpha_m}}{2}\left(\frac{4-s}{\Pi_m( {\boldsymbol{\alpha}_{\boldsymbol{n,N}}},0)} + \frac{\frac{d}{dw}\Pi_m( {\boldsymbol{\alpha}_{\boldsymbol{n,N}}},w) \Big |_{w=0}}{\pi i \left(\Pi_m( {\boldsymbol{\alpha}_{\boldsymbol{n,N}}},0)\right)^2}\right) \widehat{A}_3\left(\frac{\alpha_m}{\beta_m},-2\tau;\tau\right),\\
\widehat{H}_{m,s}( {\boldsymbol{\alpha}_{\boldsymbol{n,N}}};\tau) &:= \widehat{F}_{m,s}( {\boldsymbol{\alpha}_{\boldsymbol{n,N}}};\tau) + \widehat{G}_{m,s}( {\boldsymbol{\alpha}_{\boldsymbol{n,N}}};\tau).
\end{align*} 
The non-holomorphic functions from \cite[Theorem 1.1]{F-K} (see \S \ref{intro}) are defined by
\begin{align}
\widehat{\mathcal H}\left( {\boldsymbol{\zeta}_{\boldsymbol{n,N}}};q \right) &= \widehat{\mathcal H}_{n,N}\left( {\boldsymbol{\zeta}_{\boldsymbol{n,N}}};q \right)
:= \frac{1}{\eta(\tau)}\left(\sum_{j=1}^N  \!\!\left(\zeta_{2\beta_j}^{-\alpha_j} \widehat{H}_{j,1}\!\left( {\boldsymbol{\alpha}_{\boldsymbol{n,N}}};\tau\right) \!-\! \zeta_{2\beta_j}^{-3\alpha_j}\widehat{H}_{j,3}\!\left( {\boldsymbol{\alpha}_{\boldsymbol{n,N}}};\tau\right)\right)\right), \label{hatcalh} \\ \widehat{\mathcal A}\left( {\boldsymbol{\zeta}_{\boldsymbol{n,N}}};q \right)  &=  \widehat{\mathcal A}_{n,N}\left( {\boldsymbol{\zeta}_{\boldsymbol{n,N}}};q \right)  := \frac{1}{\eta(\tau)}\left(\sum_{j=N+1}^{n-N} (\zeta_{2\beta_j}^{-3\alpha_j}-\zeta_{2\beta_j}^{-\alpha_j})\frac{\widehat{A}_3\left(\frac{\alpha_j}{\beta_j},-2\tau;\tau\right)}{\Pi_{j}^\dagger({\boldsymbol{\alpha}_{\boldsymbol{n,N}}})}\right).\label{hatcala}
\end{align} 
We   recall the constant (a finite product) $\Pi_{j}^\dagger({\boldsymbol{\alpha}_{\boldsymbol{n,N}}})$ is defined explicitly in \cite[(4.2)]{F-K} (where one must replace $n\mapsto j$ and $k\mapsto n$).
 The holomorphic parts $\mathcal H$ and $\mathcal A$ of the functions $\widehat{\mathcal H}$ and $\widehat{\mathcal A}$ are defined by replacing the non-holomorphic functions $\widehat{H}_{j,1}, \widehat{H}_{j,3}$ and $\widehat{A}_3$ with their respective holomorphic parts $H_{j,1}, H_{j,3}$ and $A_3$ in \eqref{hatcalh} and \eqref{hatcala} above.

%%%%%%%%%%%%%%%%%%%%%%%%%%%%%%%%%%%
%%%%%%%%%%%%%%%%%%%%%%%%%%%%%%%%%%%
%%%%%				 Quantum set				   %%%%%
%%%%%%%%%%%%%%%%%%%%%%%%%%%%%%%%%%%
%%%%%%%%%%%%%%%%%%%%%%%%%%%%%%%%%%%
\section{The quantum set}\label{quantumSet}

 We call a subset $S \subseteq \Q$ a {\em quantum set} for a function $F$ with respect to the group $G\subseteq \textnormal{SL}_2(\Z)$ if both $F(x)$ and $F(Mx)$ exist (are non-singular) for all $x\in S$ and $M\in G$. 

In this section, we will show that $\qs$ as defined in \eqref{qSetDef} is a quantum set for $\mathcal{A}_{n,N}$ and $\mathcal{H}_{n,N}$ with respect to the group $\qSubgroup$.  Recall that $\qs$ is defined as

\begin{align*} \qs := \left\{\frac{h}{k}\in \Q\; \middle\vert\; \begin{aligned} & \ h\in\Z, k\in\N, \gcd(h,k) = 1, \ \beta_j \nmid k\ \forall\ 1\le j\le n,\\&\left\vert \frac{\alpha_j}{\beta_j}k - \left[\frac{\alpha_j}{\beta_j}k\right]\right\vert > \frac{1}{6}\ \forall\ 1\le j\le n\end{aligned} \right\},\end{align*}
where $[x]$ is the closest integer to $x$.
 
We further recall that the holomorphic part of our ``modular object"  (see Section \ref{intro}) is $R_{n,N}+b_{n,N}$. We now analyze the convergence of $R_{n,N}$ and $b_{n,N}$ separately.  

It was shown in \cite[Section 3]{WIN} that $\qs$ is a quantum set for $\mathcal A_{n,N}(\tau) = q^{-\frac{1}{24}} R_{n,N}(\tau)$. Moreover, the following theorem establishes the convergence of $R_{n,N}$ on $\qs$.
 
\begin{theoremno}[\cite{WIN} Theorem 3.2] For $\boldsymbol{\zeta}_{\boldsymbol{n,N}}$ as in \eqref{zetavec}, if $\frac{h}{k} \in \qs$, then $R_n(\boldsymbol{\zeta}_{\boldsymbol{n,N}}; \zeta_k^h)$ converges and can be evaluated as a finite sum.  In particular,  we have that: 
 
\begin{multline*} %\label{eqn_Rnconvsum}
R_n(\boldsymbol{\zeta}_{\boldsymbol{n,N}}; \zeta_k^h) = \prod_{j=1}^n \frac{1}{1 - ((1-x_j^k)(1-x_j^{-k}))^{-1}}\\
\times \!\!\!\!\! \sum_{\substack{0 < m_1\le k\\ 0 \le m_2, \dots, m_n < k}} \frac{\zeta_k^{h[(m_1 + m_2 + \dots + m_n)^2 + (m_1 + \dots + m_{n-1}) + (m_1 + \dots + m_{n-2}) + \dots + m_1]}}{(x_1\zeta_k^h;\zeta_k^h)_{m_1} \left(\frac{\zeta_k^h}{x_1};\zeta_k^h\right)_{m_1} (x_2 \zeta_k^{hm_1};\zeta_k^h)_{m_2 + 1} \left(\frac{\zeta_k^{hm_1}}{x_2};\zeta_k^h\right)_{m_2+1}} \\ 
\times \frac{1}{(x_3 \zeta_k^{h(m_1 + m_2)};\zeta_k^h)_{m_3 + 1}\!\!\left(\frac{\zeta_k^{h(m_1 + m_2)}}{x_3};\zeta_k^h\!\right)_{\!m_3 + 1} \!\!\!\!\!\!\!\!\!\!\cdots(x_n \zeta_k^{h(m_1 + \dots + m_{n-1})};\zeta_k^h)_{ m_n+1} \!\!\left(\!\frac{\zeta_k^{h(m_1 + \dots + m_{n-1})}}{x_n};\zeta_k^h\!\right)_{\!  m_n+1} },
\end{multline*}   
where $\boldsymbol{\zeta}_{\boldsymbol{n,N}} = (x_1, x_2, \dots, x_n)$.
\end{theoremno}

We now turn our attention to $b_{n,N}$. In what follows, as in the definition of $\qs$, we take $h\in\Z$, $k\in\N$ such that $\gcd(h,k) = 1$. 
To show that $b_n(\boldsymbol{\zeta}_{\boldsymbol{n,N}}; \zeta)$ is defined for $\zeta = e^{2\pi i h/k}$ with $\frac{h}{k} \in\qs$, it is enough to show that 
\begin{equation}\label{bnPiece}\frac{1}{(\zeta)_\infty}A_3\left(x_j, -\frac{2h}{k}; \frac{h}{k}\right)\end{equation}
is defined, with $A_3(u, v; \tau)$ as in \eqref{A3def}.  For this proof, we will make use of the following transformation formula for $A_3(u, v; \tau)$.

\begin{proposition}\label{AtranslateThm} For $u, v\in\C$, $\tau\in\mathbb{H}$ we have that
\begin{equation}\label{Atranslate}A_3(u, v+\tau; \tau) = e^{-2\pi i u}A_3(u, v;\tau) + ie^{\pi i u-\pi i v - 3\pi i \tau/4}\vartheta(v; 3\tau),\end{equation}
where $\vartheta$ is as defined in \eqref{thetaDef}.\end{proposition}

\begin{proof}
To prove our desired transformation formula, we will first rewrite $\mathscr{R}_3(u, v+\tau; \tau)$ in terms of $\mathscr{R}_3(u, v;\tau)$.  By definition \eqref{AminusDef}
\begin{align}
	\mathscr{R}_3(u, v + \tau; \tau) = &\frac{i}{2}\vartheta(v + \tau; 3\tau) R(3u - v - \tau; 3\tau) + \frac{i}{2}e^{2\pi i u}\vartheta(v + 2\tau; 3\tau) R(3u - v - 2\tau; 3\tau)\nonumber\\
	&+ \frac{i}{2}e^{4\pi i u}\vartheta(v + 3\tau; 3\tau)R(3u - v - 3\tau; 3\tau).\label{toTranslate}
\end{align}
Letting $\tau\mapsto3\tau$ in Proposition \ref{thetaTransform}, we can rewrite the third summand \eqref{toTranslate} as
\begin{align*}
	\frac{i}{2}e^{4\pi i u}\vartheta(v + 3\tau; 3\tau)R(3u - v - 3\tau; 3\tau) &= -\frac{i}{2}e^{-3\pi i \tau - 2\pi i v + 4\pi i u}\vartheta(v; 3\tau)R(3u - v - 3\tau; 3\tau)
\end{align*}
We now let $\tau\mapsto3\tau$ and $u\mapsto 3u - v - 3\tau$ in the second transformation in Proposition \ref{Rtransform} to obtain
\begin{align}
	-\frac{i}{2}&e^{-3\pi i \tau - 2\pi i v + 4\pi i u}\vartheta(v; 3\tau)R(3u - v - 3\tau; 3\tau)\nonumber\\
	&= -i e^{\pi i u - \pi i v - 3\pi i \tau/4}\vartheta(v; 3\tau) + \frac{i}{2} e^{-2\pi i u}\vartheta(v; 3\tau)R(3u - v; 3\tau)\label{finishTranslate}
\end{align}
Plugging \eqref{finishTranslate} in for \eqref{toTranslate} and using the definition of $\mathscr{R}_3(u, v;\tau)$, we see that
\begin{align}
	\mathscr{R}_3(u, v+\tau; \tau) 	&= e^{-2\pi i u}\mathscr{R}_3(u, v; \tau) - i e^{\pi i u - \pi i v - 3\pi i \tau/4}\vartheta(v; 3\tau).\label{AminusTransform}
\end{align}
By Theorem \ref{completeAtransform}, we have that
\[A_3(u, v+\tau; \tau) +\mathscr{R}_3(u, v+\tau; \tau) = e^{-2\pi i u}A_3(u, v;\tau) + e^{-2\pi i u}\mathscr{R}_3(u, v;\tau).\]
Using \eqref{AminusTransform}, we then achieve the desired result.
\end{proof}

We are now ready to show that the function in \eqref{bnPiece} converges.

\begin{theorem}\label{A3conv} For $\frac{h}{k}\in\qs$, $\zeta=e^{2\pi i h/k}$, and $x_j = e^{2\pi i \alpha_j/\beta_j}$ the $j$-th component in $\boldsymbol{\zeta}_{\boldsymbol{n,N}}$ (as in \eqref{zetavec}),  the function $\frac{1}{(\zeta)_\infty} A_3\left(\frac{\alpha_j}{\beta_j}, - \frac{2h}{k}; \frac{h}{k}\right)$ converges.  In particular,

\begin{align*}
	\frac{1}{(\zeta)_\infty} &A_3\left(\frac{\alpha_j}{\beta_j}, - \frac{2h}{k}; \frac{h}{k}\right) = \frac{x_j^{5/2}}{1- x_j} R_1(x_j; \zeta) + x_j^{3/2}\\
	&= \frac{x_j^{5/2}}{(1-x_j)(1 - ((1-x_j^k)(1-x_j^{-k}))^{-1})}\sum_{0\le s < k} \frac{\zeta^{s^2}}{(x_j\zeta; \zeta)_s (x_j^{-1}\zeta; \zeta)_s} + x_j^{3/2}.
\end{align*}
\end{theorem}

\begin{proof}
We start with equations (1.7) to (1.10) in Andrews which show that for $u \in\C$, $\tau\in\mathbb{H}$, and $q = e^{2\pi i \tau}$,
\begin{align*}
	R_1(e^{2\pi i u}; q) &= \sum_{m = 0}^\infty \frac{q^{m^2}}{(e^{2\pi i u}q; q)_m (e^{-2\pi i u}q; q)_m} = \frac{1-e^{2\pi iu}}{(q)_\infty}\sum_{m\in\Z} \frac{(-1)^m q^{m(3m+1)/2}}{1 - e^{2\pi i u}q^m}\\
	&= \frac{1- e^{2\pi i u}}{(q)_\infty} e^{-3\pi i u} A_3(u, -\tau; \tau).
\end{align*}
Taking $v = -2\tau$ in \eqref{Atranslate} and rearranging, this gives us
\begin{equation}\label{A3asR1} \frac{1}{(q)_\infty} A_3(u, -2\tau; \tau) = \frac{e^{5\pi i u}}{1 - e^{2\pi i u}}R_1(e^{2\pi i u}; q) - \frac{i e^{3\pi i u}e^{5\pi i \tau/4}}{(q)_\infty} \vartheta(-2\tau; 3\tau).\end{equation}
Applying the Jacobi triple product from Proposition \ref{thetaTransform}, we can simplify the second term in \eqref{A3asR1} as 
\begin{align*}
	\frac{i e^{3\pi i u}e^{5\pi i \tau / 4}}{(q)_\infty} \vartheta(-2\tau; 3\tau) 	
	&= \frac{e^{3\pi i u} e^{4\pi i \tau}}{(q)_\infty} \prod_{m=1}^\infty (1-q^{3m})(1-q^{3m-5})(1-q^{3m+2})\\
	&= \frac{e^{3\pi i u}q^2}{(1-q)(1-q^2)\prod_{m=3}^\infty (1-q^m)}(1-q^{-2})(1-q)\prod_{m=3}^\infty (1-q^m)\\
	&= \frac{e^{3\pi i u}q^2(1-q^{-2})}{(1-q^2)} = -e^{3\pi i u}.
\end{align*}
This simplification allows us to see that
\[\frac{1}{(q)_\infty} A_3(u, -2\tau; \tau) = \frac{e^{5\pi i u}}{1- e^{2\pi i u}} R_1(e^{2\pi i u}; q) + e^{3\pi i u}.\]
In Theorem \cite[Theorem 3.2]{WIN} above,   it was shown  that $R_1(e^{2\pi i u}; \zeta)$ is defined for $u = \frac{\alpha_j}{\beta_j}$ and $\zeta = e^{2\pi i h/k}$ with $\frac{h}{k}\in\qs$.  By definition, $\frac{\alpha_j}{\beta_j} \not\in \Z$, meaning $1 - x_j \neq 0$.  Therefore, we have shown that $\frac{1}{(\zeta)_\infty} A_3\left(\frac{\alpha_j}{\beta_j}, - \frac{2h}{k}; \frac{h}{k}\right)$ is defined for $\frac{h}{k} \in \qs$ as desired.

To obtain the exact formula for $\frac{1}{(\zeta)_\infty} A_3\left(\frac{\alpha_j}{\beta_j}, - \frac{2h}{k}; \frac{h}{k}\right)$, we let $n = 1$ in the exact formula given in \cite[Theorem 3.2]{WIN} above.

\end{proof}

We obtain the following corollary from Theorem \ref{A3conv}. 

\begin{cor}\label{bkDefined} The function $b_n(\boldsymbol{\zeta}_{\boldsymbol{n,N}}; \zeta)$ is defined for $\zeta = e^{2\pi i h/k}$ with $\frac{h}{k} \in\qs$.\end{cor}

We also obtain the following corollary, which we will need in the proof of Theorem \ref{thm_main_NN}. 

\begin{cor}\label{HAconv}
For $x\in\qs$, the functions $\mathcal{H}_{n,N}(x)$ and $\mathcal{A}_{n,N}(x)$ converge.
\end{cor}

\begin{proof}
First we consider $\mathcal{A}_{n,N}(x)$. From  \eqref{hatcala}, we have
\[
\mathcal{A}_{n,N}(\tau)=\frac{1}{\eta(\tau)}\sum_{j=N+1}^{n-N}\frac{(\zeta_{2\beta_j}^{-3\alpha_j} - \zeta_{2\beta_j}^{-\alpha_j})}{\displaystyle\Pi^\dag_j( {\bs{\alpha_{n,N}}})} A_3\left(\frac{\alpha_j}{\beta_j},-2\tau;\tau\right),
\]
a linear combination of $A_3(\frac{\alpha_j}{\beta_j},-2\tau;\tau) /\eta(\tau)$. Thus, the convergence of $\mathcal{A}_{n,N}(x)$ for $x\in\qs$   follows directly from Theorem \ref{A3conv}.

In order to show the convergence of  $\mathcal{H}_{n,N}(x)$,  we consider the holomorphic part of $\widehat{B}_{n,N}(\tau)$ at $x\in\qs$, namely
\begin{equation}\label{Bnhol}
e^{-\frac{\pi i x}{12}}B_{n,N}^+(x)=\mathcal{H}_{n,N}(x)+\mathcal{A}_{n,N}(x)=\zeta^{-\frac{1}{24}}\left(R_n(\boldsymbol{\zeta}_{\boldsymbol{n,N}}; \zeta)+b_n(\boldsymbol{\zeta}_{\boldsymbol{n,N}}; \zeta)\right),
\end{equation}
where $\zeta = e^{2\pi i x}$.  From Theorem \cite[Theorem 3.2]{WIN} above and Corollary \ref{bkDefined}, we have that the left-hand side of \eqref{Bnhol} converges on $\qs$, which yields the claim.

\end{proof}

 %%%%%%%%%%%%%%%%%%%%%%%%%%%%%%%%%%%
%%%%%%%%%%%%%%%%%%%%%%%%%%%%%%%%%%%
%%%%%				       N > 0					   %%%%%
%%%%%%%%%%%%%%%%%%%%%%%%%%%%%%%%%%%
%%%%%%%%%%%%%%%%%%%%%%%%%%%%%%%%%%%

 \section{Proof of Theorem \ref{thm_main_NN}} \label{nPosProof}
 Let $\gamma = \sm{a}{b}{c}{d}\in\qSubgroup$ and $x\in\qs$ as before.    
From \cite[Theorem 1.1]{F-K} (see   \S \ref{intro}), we deduce that
\begin{multline}\label{N_quant}
\left(\mathcal H_{n,N}( x) - \chi_\gamma(cx +d)^{-\frac32}\mathcal H_{n,N}(\gamma x)\right) + \left(\mathcal A_{n,N}( x) -\chi_\gamma(cx+d)^{-\frac12}\mathcal A_{n,N}(\gamma x)\right)
\\ 
= -\left(\mathcal H_{n,N}^-( x) - \chi_\gamma(cx+d)^{-\frac32} \mathcal H_{n,N}^-(\gamma x)\right) - \left(\mathcal A_{n,N}^-( x) -\chi_\gamma(cx+d)^{-\frac12}\mathcal A_{n,N}^-(\gamma x)\right),
\end{multline}
  where we write $\mathcal{H}_{n,N}^-$ and $\mathcal{A}_{n,N}^-$ to denote the non-holomorphic parts of the functions $\widehat{\mathcal{H}}_{n,N}$ and $\widehat{\mathcal{A}}_{n,N}$, respectively (see \eqref{hatcalh} and \eqref{hatcala}).   

In this section, we prove $\mathcal H_{n,N}+\mathcal A_{n,N}$ is a mixed weight quantum modular form on $\qs$. To do so, we first  show that the left-hand side of \eqref{N_quant} is defined on $\qs$. This follows directly from Corollary \ref{HAconv}.  Therefore, to prove Theorem \ref{thm_main_NN}, it remains to be seen that the right-hand side of \eqref{N_quant} extends to an analytic function in $\R-\{\frac{-c}{d}\}$.

It is shown in \cite[Theorem 1.7]{WIN} above that 
$\mathcal A_{n,N}^-( x) -\chi_\gamma(cx+d)^{-\frac12}\mathcal A_{n,N}^-(\gamma x)$ is analytic in $x$ on $\R-\{\frac{-c}{d}\}$.  Thus, we turn to the function $\mathcal H_{n,N}^-$.  
 A short calculation using the definition of $\mathcal H_{n,N}^-$, as well as \eqref{hatcalh} and \cite[eq. (4.5), (4.6), and (4.18)]{F-K}, leads to the following result.

 \begin{lemma}\label{lem:H^-} With notation and hypotheses as above, we have that
\begin{align*}
\mathcal H_{n,N}^-(\tau) & = \frac{2}{\eta(\tau)} \sum_{j=1}^N \frac{\zeta_{2\beta_j}^{-3\alpha_j} - \zeta_{2\beta_j}^{-5\alpha_j}}{\Pi_j(0)} \mathscr R_3\left(\frac{\alpha_j}{\beta_j}\right) - \frac{1}{2\pi i \eta(\tau)} \sum_{j=1}^N \frac{ \zeta_{2\beta_j}^{-3\alpha_j} - \zeta_{2\beta_j}^{-5\alpha_j}}{\Pi_j(0)}
 \left.\frac{d}{du}\mathscr R_3(u) \right|_{u= \frac{\alpha_j}{\beta_j}} \\
  & = 2 \sum_{j=1}^N \frac{\zeta_{2\beta_j}^{-3\alpha_j} - \zeta_{2\beta_j}^{-5\alpha_j}}{\Pi_j(0)} S\left(\frac{\alpha_j}{\beta_j};\tau\right) - \frac{1}{2\pi i } \sum_{j=1}^N \frac{\zeta_{2\beta_j}^{-3\alpha_j} - \zeta_{2\beta_j}^{-5\alpha_j}}{\Pi_j(0)}
 \left.\frac{d}{du}S(u;\tau) \right|_{u= \frac{\alpha_j}{\beta_j}}, 
 \end{align*}
 where  $$S(u;\tau) :=\frac{ \mathscr R_3(u,-2\tau;\tau)}{\eta(\tau)}.$$
 \end{lemma}

 In order to examine $\mathcal H_{n,N}^-(\gamma x)$, we transform the functions $S(\frac{\alpha}{\beta};\tau)$ and  $\frac{d}{du}S(u;\tau)|_{u= \frac{\alpha}{\beta}}$ separately. For ease of notation, we sometimes suppress dependence on $j$, and write, for example, $\frac{\alpha}{\beta}$ for $\frac{\alpha_j}{\beta_j}$  when the context is clear. Note that it  suffices to consider the generators $T$ and $S_\ell$ of $\qSubgroup$ as before. 
 Using the definition of $\mathscr{R}_3$ in \eqref{AminusDef},    we begin by rewriting $S(u;\tau)$ as 
\begin{align}
S(u;\tau) &=   \frac{i}{2\eta(\tau)} \sum_{j=0}^2 e(ju) \vartheta((j-2)\tau;3\tau)R(3u+(2-j) \tau;3\tau) \nonumber\\
&=\frac{q^{-\frac16}}{2} \sum_{\pm}\mp e(u(2\mp 1)) R(3u\pm \tau;3\tau) - q^{-\frac{1}{24}}e\left(\frac32 u\right). \label{Sdef}
\end{align}
 
The second equality follows directly from Proposition \ref{thetaTransform} (3) and Proposition \ref{Rtransform} (2). More precisely, from Proposition \ref{thetaTransform} (3) we have $\vartheta(-2\tau;3\tau)=iq^{\frac23}\eta(\tau),\ \vartheta(-\tau;3\tau)=iq^{-\frac16}\eta(\tau)$, and $\vartheta(0;3\tau)=0$, and from Proposition \ref{Rtransform} (2) $R(3u+2\tau;3\tau)=2e(\frac32u)q^{\frac58}-e(3u)q^{\frac12} R(3u-\tau;3\tau)$.

We deduce the following transformation properties of $S(\frac{\alpha}{\beta};\tau)$.

 \begin{lemma}\label{lem_Stransform}  With notation and hypotheses as above, we have that
 \begin{align*}
 S\left(\frac{\alpha}{\beta};\tau+1\right) &= \zeta_{24}^{-1} S\left(\frac{\alpha}{\beta};\tau\right), \\
 S\left(\frac{\alpha}{\beta};S_\ell\tau\right) &= (\ell\tau+1)^{\frac12}\zeta_{24}^\ell S\left(\frac{\alpha}{\beta};\tau\right) + \frac{1}{2}(\ell\tau+1)^{\frac12}\zeta_{24}^\ell e\left(2\frac{\alpha}{\beta}\right) H_{\alpha, \beta}(\tau) + \mathcal E_1\left(\frac{\alpha}{\beta},\ell;\tau\right),
 \end{align*} where 
  $\mathcal E_1(\frac{\alpha}{\beta},\ell;\tau) := (\ell\tau+1)^{\frac12} \zeta_{24}^\ell q^{-\frac{1}{24}} e(\frac32 \frac{\alpha}{\beta}) - e(-\frac{S_\ell \tau}{24} ) e(\frac32 \frac{\alpha}{\beta}).$  The function $H_{\alpha,\beta}(\tau)$ is defined in \eqref{def_Hab}, that is,
 \begin{equation}\label{Hperiodint} H_{\alpha,\beta}(\tau) = \sqrt{3}  \sum_{\pm}\mp e\left(\mp\frac{1}{6}\right) \int_{\frac{1}{\ell}}^{i\infty} \frac{g_{\pm\frac{1}{3}+\frac{1}{2}, -3\frac{\alpha}{\beta}+\frac{1}{2}}(3\rho)}{\sqrt{-i(\rho +\tau)}} \ d\rho. 
 \end{equation}
 \end{lemma}

\begin{proof}
Letting $\tau\mapsto \tau+1$ in \eqref{Sdef} gives
\begin{equation}\label{S-Ttransform}
 S(u;\tau+1)= \frac{\zeta_6^{-1}q^{-\frac16}}{2}\sum_{\pm }\mp e(u(2\mp 1)) R(3u\pm \tau\pm1;3\tau+3) - \zeta_{24}^{-1}q^{-\frac{1}{24}}e\left(\frac32 u\right).
\end{equation}
Using the transformation properties (1) and (4) in Proposition \ref{Rtransform}, we have
\[
R(3u\pm \tau\pm1;3\tau+3)=-e^{-\frac{3\pi i}{4}}R(3u\pm\tau;3\tau).
\]
Substituting this into \eqref{S-Ttransform} yields
\[
 S(u;\tau+1)= \frac{\zeta_{24}^{-1}q^{-\frac16}}{2}\sum_{\pm}\mp e(u(2\mp 1)) R(3u\pm \tau;3\tau) - \zeta_{24}^{-1}q^{-\frac{1}{24}}e\left(\frac32 u\right)
=\zeta_{24}^{-1} S(u;\tau).
\]

We now turn to the $S_\ell$ transformation. Recalling the definition of $F_{\alpha,\beta}$ in  \cite{WIN}
\[%\label{def_Fab}
F_{\alpha,\beta}(\tau):= q^{-\frac16} \sum_{\pm}  \pm e \left(\mp \frac{\alpha}{\beta} \right) R\left(\frac{3\alpha}{\beta} \pm \tau; 3\tau \right),
\]
we rewrite  $S(\frac{\alpha}{\beta};\tau)$ in terms of $F_{\alpha,\beta}(\tau)$ as
\[
S\left(\frac{\alpha}{\beta};\tau\right) =-\frac{1}{2}e\left(\frac{2\alpha}{\beta}\right)F_{\alpha,\beta}(\tau) -q^{-\frac{1}{24}}e\left(\frac32\frac{\alpha}{\beta}\right),
\]
and thus 
\begin{equation}\label{S-Stransform}
S\left(\frac{\alpha}{\beta};S_\ell\tau\right) =-\frac{1}{2}e\left(\frac{2\alpha}{\beta}\right)F_{\alpha,\beta}(S_\ell\tau) -e\left(-\frac{S_\ell\tau}{24}\right)e\left(\frac32\frac{\alpha}{\beta}\right).
\end{equation}
 We further recall the definition of $H_{\alpha,\beta}$ from \cite{WIN} 
\begin{equation}\label{def_Hab}
H_{\alpha,\beta}(\tau) :=  F_{\alpha,\beta}(\tau) - \zeta_{24}^{-\ell}(\ell \tau +1)^{-\frac12}F_{\alpha,\beta}(S_\ell \tau).
\end{equation}
 
Inserting \eqref{def_Hab} into \eqref{S-Stransform} with a direct calculation reveals that
\begin{align*}
S\left(\frac{\alpha}{\beta};S_\ell\tau\right) & =- \frac12 (\ell\tau+1)^{\frac12}\zeta_{24}^\ell e\left(\frac{2\alpha}{\beta}\right)(F_{\alpha,\beta}(\tau)  - H_{\alpha,\beta}(\tau))-e\left(-\frac{S_\ell\tau}{24}\right)e\left(\frac32\frac{\alpha}{\beta}\right)\\
&=  (\ell\tau+1)^{\frac12}\zeta_{24}^\ell S\left(\frac{\alpha}{\beta};\tau\right)+ \frac12 (\ell\tau+1)^{\frac12}\zeta_{24}^\ell e\left(\frac{2\alpha}{\beta}\right)  H_{\alpha,\beta}(\tau)\\
&\qquad+(\ell\tau+1)^{\frac12} \zeta_{24}^\ell q^{-\frac{1}{24}} e\left(\frac32 \frac{\alpha}{\beta}\right) - e\left(-\frac{S_\ell \tau}{24} \right) e\left(\frac32 \frac{\alpha}{\beta}\right),
\end{align*}
as claimed.
\end{proof}

In order to establish the transformation properties of $\frac{d}{du}S(u;\tau)|_{u= \frac{\alpha}{\beta}}$, we first deduce the following, using \eqref{Sdef}:

 $$ \left.\frac{d}{du} S(u;\tau) \right|_{u=\frac{\alpha}{\beta}} = -3\pi i q^{-\frac{1}{24}} e\left(\frac32 \frac{\alpha}{\beta}\right)  + W_1(\tau) + W_2(\tau),$$
 where
 \begin{align}
 W_1(\tau) &:=  \frac{q^{-\frac{1}{6}}}{2} \sum_{\pm} \mp e\left(\frac{\alpha}{\beta}(2\mp 1)\right) \left.\frac{d}{du} R(3u\pm \tau;3\tau) \right|_{u=\frac{\alpha}{\beta}}, \label{def_W1}\\
 W_2(\tau) &:=  \pi i q^{-\frac{1}{6}} \sum_{\pm} (1\mp 2) e\left(\frac{\alpha}{\beta}(2\mp 1)\right)    R\left(3\frac{\alpha}{\beta} \pm \tau;3\tau\right). \label{def_W2}
 \end{align}

First we establish the following trnasformation properties of $ W_2(\tau)$.

 \begin{lemma}\label{lem_W2} With notation and hypotheses as above, we have that
\begin{align*}
W_2(\tau+1) &= \zeta_{24}^{-1} W_2(\tau)  \\
W_2(S_\ell \tau) &= (\ell\tau+1)^{\frac12} \zeta_{24}^\ell \left(W_2(\tau) + \widetilde{H}_{\alpha,\beta}(\tau)\right).
\end{align*}
 \end{lemma}

 \begin{proof}
As before, shifting $\tau\mapsto \tau+1$ in \eqref{def_W2} and using the transformation properties (1) and (4) in Proposition \ref{Rtransform} directly yields the first claim. 

On the other hand, letting $\tau\mapsto S_\ell\tau=\frac{-1}{\tau_\ell}$ in \eqref{def_W2} with $\tau_\ell=\frac{-1}{\tau}-\ell$  as before, we have  
 \begin{equation}\label{W2transform}
 W_2(S_\ell\tau) =\pi i \ e\left(\frac{1}{6\tau_\ell}\right) \sum_{\pm} (1\mp 2) e\left(\frac{\alpha}{\beta}(2\mp 1)\right)    R\left(3\frac{\alpha}{\beta} \mp \frac{1}{\tau_\ell};\frac{-3}{\tau_\ell}\right).
 \end{equation}
 From the proof of \cite[Proposition 4.1]{WIN}, we know that 
\begin{multline}\label{Rmodtransform}
 R\left(3\frac{\alpha}{\beta} \mp \frac{1}{\tau_\ell};\frac{-3}{\tau_\ell}\right)=(\ell\tau+1)^{\frac12}\zeta_{24}^\ell e\left(\frac{-1}{6\tau_\ell}-\frac{\tau}{6}\right)   R\left(3\frac{\alpha}{\beta} \pm \tau;3\tau\right) \\
 +\sqrt{3}(\ell\tau+1)^{\frac12} \zeta_{24}^\ell e\left(\frac{-1}{6\tau_\ell}\pm\frac{\alpha}{\beta}\mp\frac16\right)
 \int_{\frac{1}{\ell}}^{i\infty} \frac{g_{\pm\frac13+\frac12,-3\frac{\alpha}{\beta}+\frac12}(3\rho)}{\sqrt{-i(\rho+\tau)}} d\rho.  
\end{multline} 
Inserting \eqref{Rmodtransform} into \eqref{W2transform} gives us 
\begin{align*}
 W_2(S_\ell\tau) &=\pi i (\ell\tau+1)^{\frac12}\zeta_{24}^\ell q^{-\frac16} \sum_{\pm}(1\mp 2)e\left(\frac{\alpha}{\beta}(2\mp 1)\right) R\left(3\frac{\alpha}{\beta} \pm \tau;3\tau\right)\\
 &\qquad +\sqrt{3}\pi i (\ell\tau+1)^{\frac12}\zeta_{24}^\ell e\left(2\frac{\alpha}{\beta}\right) \sum_{\pm} (1\mp 2)e\left(\mp \frac16\right) \int_{\frac{1}{\ell}}^{i\infty} \frac{g_{\pm\frac13+\frac12,-3\frac{\alpha}{\beta}+\frac12}(3\rho)}{\sqrt{-i(\rho+\tau)}} d\rho\\
 &=(\ell\tau+1)^{\frac12} \zeta_{24}^\ell W_2(\tau)+(\ell\tau+1)^{\frac12}\zeta_{24}^\ell\widetilde{H}_{\alpha,\beta}(\tau),
\end{align*}
where 
\begin{equation}\label{def_Htilde}
\widetilde{H}_{\alpha,\beta}(\tau):=\sqrt{3}\pi i e\left(2\frac{\alpha}{\beta}\right) \sum_{\pm} (1\mp 2) e\left(\mp \frac16\right) \int_{\frac{1}{\ell}}^{i\infty} \frac{g_{\pm\frac13+\frac12,-3\frac{\alpha}{\beta}+\frac12}(3\rho)}{\sqrt{-i(\rho+\tau)}} d\rho.
\end{equation}  \end{proof}

We may now deduce the following transformation properties of $W_1(\tau)$. 
\begin{lemma}\label{lem_W1} Let $m :=[\frac{3\alpha}{\beta}]$ so that $\frac{3\alpha}{\beta}=m +r$ with $r\in(-\frac12,\frac12)$.   With notation and hypotheses as above, we have that 
\begin{align*}
 W_1(\tau+1)&= \zeta_{24}^{-1} W_1(\tau), \\
W_1(S_\ell\tau) &=   (\ell\tau+1)^{\frac32} \zeta_{24}^\ell W_1(\tau) \\
&\qquad+\pi i \ell\tau(\ell\tau+1)^\frac12  \zeta_{24}^\ell q^{-\frac16} \sum_{\pm} e\left(\frac{\alpha}{\beta}(2\mp 1)\right)R\left(3\frac{\alpha}{\beta} \pm \tau;3\tau\right) \\
&\qquad +\frac{3(-1)^m }{2}  e\left(-\frac{S_\ell\tau}{6}\right) \sum_{\pm}\mp e\left(\frac{\alpha}{\beta}(2\mp 1)\right) \left.\frac{d}{dv}\left(T_1(v;\tau) + T_2(v;\tau) \right) \ \right|_{v=r},
\end{align*}
\noindent where  $T_1$,  and $T_2$ are defined in \eqref{def_T1v} and \eqref{def_T2v} below respectively. 
\end{lemma}

\begin{proof}
The first claim follows again by letting $\tau\mapsto \tau+1$ in \eqref{def_W1} and using the transformation properties (1) and (4) in Proposition \ref{Rtransform}.

To show the second claim, we first consider
\[
\left.\frac{d}{du} R(3u\pm\tau;3\tau)\right|_{u=\frac{\alpha}{\beta}}.
\]
By the chain rule with $v=3u-m $ and Proposition \ref{Rtransform} (1), this derivative becomes 
\[
3\left.\frac{d}{dv} R(v+m \pm\tau;3\tau)\right|_{v=r}
=3(-1)^m \left.\frac{d}{dv} R(v\pm\tau;3\tau)\right|_{v=r}.
\]
We define
$$r(v;\tau) := R(v \pm \tau;3\tau),$$ 
and then transform this function by using transformation properties in Proposition \ref{Rtransform}.  
More precisely, we start with
\[
r(v;S_\ell\tau) = R\left(v \mp\frac{1}{\tau_\ell};-\frac{3}{\tau_\ell}\right),
\]
and apply (5) and (4) in Proposition \ref{Rtransform}. We then use (1) to shift the $R$ function by $-\frac{r\ell}{3}$. Note that $\frac{\alpha\ell}{\beta}\in2\Z$ by the definition of $\ell$ in \eqref{def_ell} and $r=\frac{3\alpha}{\beta}-[\frac{3\alpha}{\beta}]$, which yields $-\frac{r\ell}{3} \in2\Z$. Lastly, we apply (3) and (5) again to obtain 
\begin{equation}\label{rvgam}
\begin{split}
r(v;S_\ell\tau)  =& T_1(v;\tau) + T_2(v;\tau) \\
&+\Bigg[(\ell\tau+1)^{\frac12}  \zeta_{24}^\ell q^{-\frac{r^2\ell^2}{6}} e\left(-\ell\frac{(v(\ell\tau+1)\pm \tau)^2}{6(\ell\tau+1)}+\frac{r\ell}{3}(v(\ell\tau+1)\pm \tau)\right)\\
&\hspace{3in}\times R\left(v(\ell\tau+1)-r\ell\tau\pm\tau;3\tau\right)\Bigg], 
\end{split}
\end{equation}
where 
\begin{align}
T_1(v;\tau) 
&:= \sqrt{\frac{i}{3}\left(\frac{1}{\tau} + \ell\right)} e\left( \frac{\left(v(\ell\tau+1)\pm \tau\right)^2}{6\tau(\ell\tau+1)}\right) h\left(\frac{v\tau_\ell}{3} \mp \frac13;\frac{\tau_\ell}{3}\right)  \label{def_T1v}, \\
T_2(v;\tau) 
&:= -(\ell\tau+1)^{\frac12}  \zeta_{24}^\ell q^{-\frac{r^2\ell^2}{6}} e\left(-\ell\frac{(v(\ell\tau+1)\pm \tau)^2}{6(\ell\tau+1)}+\frac{r\ell}{3}(v(\ell\tau+1)\pm \tau)\right)\label{def_T2v}\\
&\hspace{3in}\times h\left(v(\ell\tau+1)-r\ell\tau\pm\tau;3\tau\right).\notag 
\end{align} 
Next we calculate the derivative of $r(v;S_\ell\tau)$.  To do so, we first consider the derivative of the exponential term on the right-hand side of \eqref{rvgam}. A short calculation shows that
\begin{equation}\label{dvexp}
\left.\frac{d}{dv}  e\left(-\ell\frac{(v(\ell\tau+1)\pm \tau)^2}{6(\ell\tau+1)}+\frac{r\ell}{3}(v(\ell\tau+1)\pm \tau)\right)\ \right|_{v=r} 
= \mp \frac{2\pi i \ell \tau}{3} q^{\frac{r^2\ell^2}{6}} e\left(\frac{-\ell\tau^2}{6(\ell\tau+1)} \right).
\end{equation}
We further examine the derivative of $R$ function in \eqref{rvgam}. Applying the chain rule with $u=\frac{v(\ell\tau+1)-r\ell\tau+m }{3}$ and then using Proposition \ref{Rtransform} (1) gives us
\begin{equation}\label{dvR} 
\left.\frac{d}{dv} R\left(v(\ell\tau+1)-r\ell\tau\pm \tau ;3\tau\right) \ \right|_{v=r} 
= \frac{(\ell\tau+1)}{3}(-1)^m  \left.\frac{d}{du} R(3u\pm \tau;3\tau) \ \right|_{u=\frac{\alpha}{\beta}}.
\end{equation}
Therefore, by \eqref{dvexp}, \eqref{dvR}, and a direct calculation with Proposition \ref{Rtransform} (1), we have   
\begin{equation}\label{dvrv}
\begin{split}
3(-1)^m \left.\frac{d}{dv}  r(v;S_\ell\tau) \ \right|_{v=r} =&3(-1)^m  \left.\frac{d}{dv}\left(T_1(v;\tau) + T_2(v;\tau) \right) \ \right|_{v=r} \\
&+ (\ell\tau+1)^{\frac12}  \zeta_{24}^\ell \left(\mp 2\pi i \ell\tau\right) e\left(\frac{-\ell\tau^2}{6(\ell\tau+1)} \right)R\left(\frac{3\alpha}{\beta}\pm \tau ;3\tau\right)    \\ 
& +  (\ell\tau+1)^{\frac32}  \zeta_{24}^\ell  e\left(\frac{-\ell\tau^2}{6(\ell\tau+1)} \right)  \left.\frac{d}{du} R(3u\pm \tau;3\tau) \ \right|_{u=\frac{\alpha}{\beta}}.
\end{split}
\end{equation} 

We are now able to prove the second claim of the lemma.  By the definition of $W_1$ in \eqref{def_W1}, and \eqref{dvrv}, we find that
\begin{align*}
W_1(S_\ell\tau) =& \frac12 e\left(-\frac{S_\ell\tau}{6}\right)\sum_{\pm}\mp e\left(\frac{\alpha}{\beta}(2\mp 1)\right) \left(3(-1)^m \left.\frac{d}{dv}  r(v;S_\ell\tau) \ \right|_{v=r}\right) \\
=&  \frac{3(-1)^m }{2}  e\left(-\frac{S_\ell\tau}{6}\right) \sum_{\pm}\mp e\left(\frac{\alpha}{\beta}(2\mp 1)\right) \left.\frac{d}{dv}\left(T_1(v;\tau) + T_2(v;\tau) \right) \ \right|_{v=r} \\
&+  \pi i \ell\tau(\ell\tau+1)^{\frac12}  \zeta_{24}^\ell q^{-\frac{1}{6}}\sum_{\pm}  e\left(\frac{\alpha}{\beta}(2\mp 1)\right) R\left(3\frac{\alpha}{\beta}\pm \tau ;3\tau\right)\\
& +   (\ell\tau+1)^{\frac32}  \zeta_{24}^\ell W_1(\tau). 
\end{align*}
\end{proof} 
We require the following Lemma.  
\begin{lemma} \label{lem_TvPeriod} Suppose that $v \in (-\frac12,\frac12)$, and $|v-r|<\epsilon$ for some sufficiently small $\epsilon>0$.    Then we have that 
\begin{align*}
 e\left(-\frac{S_\ell\tau}{6}\right) &T_1(v;\tau) \\ &=\sqrt{3}(\ell\tau+1)^{\frac12}\zeta_{24}^\ell  e\left(\mp\frac16\pm \frac{v}{3} - \frac{v\ell }{6} + \frac{v^2\ell}{6}\right) \int_{\frac{1}{\ell}}^0 \frac{g_{\frac{v\ell}{3} \pm \frac13 + \frac12,-v + \frac12}(3\rho)}{\sqrt{-i(\rho + \tau)}} d\rho, \\
 e\left(-\frac{S_\ell\tau}{6}\right) &T_2(v;\tau) \\ &=\sqrt{3}(\ell\tau+1)^{\frac12}\zeta_{24}^\ell e\left(\mp\frac16\pm \frac{v}{3} - \frac{v\ell }{6} + \frac{v^2\ell}{6}\right) \int^{i\infty}_0 \frac{g_{\frac{v\ell}{3} \pm \frac13 + \frac12,-v + \frac12}(3\rho)}{\sqrt{-i(\rho + \tau)}} d\rho. 
\end{align*}
\end{lemma}

\begin{proof} 
 Since $v, \pm \frac13 \in (-\frac12,\frac12)$,
 we may apply Theorem \ref{thm_Zh2} to  the function $h(\frac{v\tau_\ell}{3} \mp \tfrac13;\frac{\tau_\ell}{3})$ in the definition of $T_1(v;\tau)$ in \eqref{def_T1v}.   Using Proposition \ref{prop_Zg} (3) and (4), we proceed as in  \cite[(4.6) and (4.7)]{WIN}  with $v$ instead of $r$. A straightforward calculation yields the first equality stated in Lemma \ref{lem_TvPeriod}.

Similarly, the second equality follows directly by applying Theorem \ref{thm_Zh2} to  the function
$h(v(\ell\tau+1)-r\ell\tau\pm  \tau;3\tau)$ in $T_2(v;\tau)$ defined in \eqref{def_T2v} with $a=\frac{\ell}{3} (v-r)\pm \frac13,\ b=-v$, and $\tau \mapsto 3\tau$.  This is allowable because $-v \in (-\frac12,\frac12)$ and since $|v-r|<\epsilon$ for sufficiently small $\epsilon>0$, we have that $\frac{\ell}{3} (v-r)\pm \frac13 \in (-\frac12,\frac12)$.

\end{proof} 

We define  (in parallel to $H_{\alpha,\beta}(\tau)$) 
\begin{align}\label{def_Dab}
D_{\alpha,\beta}(\tau) := \sqrt{3} \sum_{\pm}  \mp e\left(\mp\frac{1}{6}\right)  \int_{\frac{1}{\ell}}^{i\infty} \frac{\left.\frac{d}{du}g_{\ell u \pm \frac13 + \frac12,-3u + \frac12}(3\rho)\right|_{u=\frac{\alpha}{\beta}}}{\sqrt{-i(\rho +\tau)}}d\rho.
\end{align} 
By the lemmas above, we finally have the following result.
\begin{proposition}\label{prop_hmintsf} 
Assume the notation and hypotheses as above. Then we have
\[
\mathcal H_{n,N}^-(\tau+1) - \zeta_{24}^\ell\mathcal H_{n,N}^-(\tau) =0,
\]
and
\begin{equation}\label{eqn_hmin_inv}
\begin{split}
&\mathcal H_{n,N}^-(S_\ell\tau) - (\ell\tau+1)^{\frac32} \zeta_{24}^\ell\mathcal H_{n,N}^-(\tau) \\ \nonumber
&=\sum_{j=1}^N \frac{\zeta_{2\beta_j}^{\alpha_j} - \zeta_{2\beta_j}^{-\alpha_j}}{2\Pi_j(0)}   \Bigg[ (\ell\tau + 1)^{\frac12}\zeta_{24}^\ell\left(\left(\frac{\ell}{2}-3\frac{\alpha_j}{\beta_j}\ell\right)H_{\alpha_j,\beta_j}(\tau) 
  -\frac{1}{2\pi i}D_{\alpha_j,\beta_j}(\tau)\right) 
+ \mathcal E_2\left(\frac{\alpha_j}{\beta_j},\ell;x\right)\Bigg], 
 \end{split}
\end{equation}
where $\mathcal E_2(\frac{\alpha}{\beta},\ell;\tau) := (\ell\tau+1)^{\frac32} \zeta_{24}^\ell q^{-\frac{1}{24}}e\big(-\frac{\alpha}{2\beta}\big) - e\left(-\frac{S_\ell \tau}{24}\right)e\big(-\frac{\alpha}{2\beta}\big)$.   
\end{proposition}

\begin{proof}
We begin by recalling Lemma \ref{lem:H^-}, that is,
\begin{equation}\label{eq:H^-}
\mathcal H_{n,N}^-(\tau)= 2\sum_{j=1}^N \frac{\zeta_{2\beta_j}^{-3\alpha_j} - \zeta_{2\beta_j}^{-5\alpha_j}}{\Pi_j(0)}S\left(\frac{\alpha_j}{\beta_j};\tau\right) - \frac{1}{2\pi i} \sum_{j=1}^N \frac{\zeta_{2\beta_j}^{-3\alpha_j} - \zeta_{2\beta_j}^{-5\alpha_j}}{\Pi_j(0)}
\left.\frac{d}{du} S(u;\tau) \right|_{u= \frac{\alpha_j}{\beta_j}},
\end{equation}
where
 $$ \left.\frac{d}{du} S(u;\tau) \right|_{u=\frac{\alpha}{\beta}} = -3\pi i q^{-\frac{1}{24}} e\left( \frac{3\alpha}{2\beta}\right)  + W_1(\tau) + W_2(\tau).$$
The first claim follows directly from Lemmas \ref{lem_Stransform}, \ref{lem_W2}, and \ref{lem_W1}. 

For the second claim, we first rewrite $W_1(S_\ell\tau)$ in Lemma \ref{lem_W1} using $S(\frac{\alpha}{\beta};\tau)$ and $W_2(\tau)$, that is, 
\begin{multline}\label{eq:rewrite_W1S} W_1(S_\ell\tau) =   (\ell\tau+1)^{\frac32}\zeta_{24}^\ell W_1(\tau) + \ell\tau (\ell\tau+1)^{\frac12} \zeta_{24}^\ell\left(W_2(\tau)-4\pi i S\left(\frac{\alpha}{\beta};\tau\right)-4\pi i q^{-\frac{1}{24}}e\left(\frac32\frac{\alpha}{\beta}\right)\right)\\
+\frac{3(-1)^m }{2}  e\left(-\frac{S_\ell\tau}{6}\right) \sum_{\pm}\mp e\left(\frac{\alpha}{\beta}(2\mp 1)\right) \left.\frac{d}{dv}\left(T_1(v;\tau) + T_2(v;\tau) \right) \ \right|_{v=r}.
\end{multline}
We now consider $\mathcal H_{n,N}^-(S_\ell\tau) $. Combining with \eqref{eq:H^-}, \eqref{eq:rewrite_W1S} and the second claims of Lemmas \ref{lem_Stransform} and \ref{lem_W2},  we have
\begin{align*}
\mathcal H_{n,N}^-&(S_\ell\tau) \\
=& 2\sum_{j=1}^N \frac{\zeta_{2\beta_j}^{-3\alpha_j} - \zeta_{2\beta_j}^{-5\alpha_j}}{\Pi_j(0)} \left(  (\ell\tau+1)^{\frac12}\zeta_{24}^\ell \left(S\left(\frac{\alpha_j}{\beta_j};\tau\right) + \frac{1}{2} e\left(2\frac{\alpha_j}{\beta_j}\right) H_{\alpha_j,\beta_j}(\tau)\right) + \mathcal E_1\left(\frac{\alpha_j}{\beta_j},\ell;\tau\right)  \right)\\
&- \frac{1}{2\pi i} \sum_{j=1}^N \frac{\zeta_{2\beta_j}^{-3\alpha_j} - \zeta_{2\beta_j}^{-5\alpha_j}}{\Pi_j(0)}
  \Bigg[-3\pi i e\left(-\frac{S_\ell\tau}{24}\right)e\left(\frac32\frac{\alpha_j}{\beta_j}\right)
 -4\pi i \ell\tau (\ell\tau+1)^{\frac12} \zeta_{24}^\ell q^{-\frac{1}{24}}e\left(\frac32\frac{\alpha_j}{\beta_j}\right)\\
& \hspace{1in}+ (\ell\tau+1)^{\frac32}  \zeta_{24}^\ell W_1(\tau)+ (\ell\tau+1)^{\frac32} \zeta_{24}^\ell W_2(\tau)-4\pi i \ell\tau (\ell\tau+1)^{\frac12} \zeta_{24}^\ell S\left(\frac{\alpha_j}{\beta_j};\tau\right)\\
& \hspace{1in} +\frac{3(-1)^{m _j}}{2}  e\left(-\frac{S_\ell\tau}{6}\right) \sum_{\pm}\mp e\left(\frac{\alpha}{\beta}(2\mp 1)\right) \left.\frac{d}{dv}\left(T_1(v;\tau) + T_2(v;\tau) \right) \ \right|_{v=r_j} \\
&\hspace{4.4in} + (\ell\tau+1)^{\frac12} \zeta_{24}^\ell \widetilde{H}_{\alpha_j,\beta_j}(\tau)\Bigg]\\ 
=&2\sum_{j=1}^N \frac{\zeta_{2\beta_j}^{-3\alpha_j} - \zeta_{2\beta_j}^{-5\alpha_j}}{\Pi_j(0)} \Bigg[  (\ell\tau+1)^{\frac32}\zeta_{24}^\ell S\left(\frac{\alpha_j}{\beta_j};\tau\right) 
+\frac{1}{2} (\ell\tau+1)^{\frac12}\zeta_{24}^\ell e\left(2\frac{\alpha_j}{\beta_j}\right) H_{\alpha_j,\beta_j}(\tau) \\
&\hspace{2in}+\frac14(\ell\tau+1)^{\frac32}\zeta_{24}^\ell q^{-\frac{1}{24}}e\left(\frac32\frac{\alpha_j}{\beta_j}\right) -\frac14 e\left(-\frac{S_\ell\tau}{24}\right)e\left(\frac32\frac{\alpha_j}{\beta_j}\right)\Bigg]\\
&- \frac{1}{2\pi i} \sum_{j=1}^N \frac{\zeta_{2\beta_j}^{-3\alpha_j} - \zeta_{2\beta_j}^{-5\alpha_j}}{\Pi_j(0)}
  \Bigg[-3\pi i(\ell\tau+1)^{\frac32}\zeta_{24}^\ell q^{-\frac{1}{24}}e\left(\frac32\frac{\alpha_j}{\beta_j}\right)
  + (\ell\tau+1)^{\frac32}  \zeta_{24}^\ell W_1(\tau) \\
  &\hspace{2in}+ (\ell\tau+1)^{\frac32} \zeta_{24}^\ell W_2(\tau)   
	 + (\ell\tau+1)^{\frac12}\zeta_{24}^\ell \widetilde{H}_{\alpha_j,\beta_j}(\tau) \\
&\hspace{1in}+\frac{3(-1)^{m _j}}{2}  e\left(-\frac{S_\ell\tau}{6}\right) \sum_{\pm}\mp e\left(\frac{\alpha}{\beta}(2\mp 1)\right) \left.\frac{d}{dv}\left(T_1(v;\tau) + T_2(v;\tau) \right) \ \right|_{v=r_j}\Bigg]\\
=& (\ell\tau+1)^{\frac32}\zeta_{24}^\ell  \mathcal{H}_{n,N}^-(\tau)\\
&+\sum_{j=1}^N \frac{\zeta_{2\beta_j}^{-3\alpha_j} - \zeta_{2\beta_j}^{-5\alpha_j}}{\Pi_j(0)}  \Bigg[ 
(\ell\tau+1)^{\frac12}\zeta_{24}^\ell \left( e\left(2\frac{\alpha_j}{\beta_j}\right) H_{\alpha_j,\beta_j}(\tau) -\frac{1}{2\pi i}  \widetilde{H}_{\alpha_j,\beta_j}(\tau)\right) \\
&\hspace{2in}+\frac12(\ell\tau+1)^{\frac32}\zeta_{24}^\ell q^{-\frac{1}{24}}e\left(\frac32\frac{\alpha_j}{\beta_j}\right) -\frac12 e\left(-\frac{S_\ell\tau}{24}\right)e\left(\frac32\frac{\alpha_j}{\beta_j}\right)\\
&\hspace{1in}- \frac{3(-1)^{m _j}}{4\pi i}  e\left(-\frac{S_\ell\tau}{6}\right) \sum_{\pm}\mp e\left(\frac{\alpha}{\beta}(2\mp 1)\right) \left.\frac{d}{dv}\left(T_1(v;\tau) + T_2(v;\tau) \right) \ \right|_{v=r_j}\Bigg].
\end{align*}
We continue to simplify the term in the parenthesis $[\ ]$ above. Using \eqref{Hperiodint} and \eqref{def_Htilde}, we have 
\[
e\left(2\frac{\alpha}{\beta}\right) H_{\alpha,\beta}(\tau) -\frac{1}{2\pi i}  \widetilde{H}_{\alpha,\beta}(\tau)=\frac{\sqrt{3}}{2}  e\left(2\frac{\alpha}{\beta}\right) \sum_{\pm}e\left(\mp \frac16\right) \int_{\frac{1}{\ell}}^{i\infty} \frac{g_{\pm\frac13+\frac12,-3\frac{\alpha}{\beta}+\frac12}(3\rho)}{\sqrt{-i(\rho+\tau)}} d\rho.
\] 
Moreover, since we  take the derivative in $v$ at the points $v=r_j$ in what follows, and $r_j \in(-\frac12,\frac12)$ we may assume $|v-r_j|<\epsilon$ for sufficiently small $\epsilon>0$.   We further note from Proposition \ref{prop_Zg} (2) that 
for $m\in\Z$,
\[
g_{a,b} = e(ma) g_{a,b-m}.
\]
Applying this to Lemma \ref{lem_TvPeriod}, we obtain  
\begin{multline}\label{eq:TvPeriod}
3(-1)^m  e\left(-\frac{S_\ell\tau}{6}\right)\left[T_1(v;\tau) + T_2(v;\tau)\right]\\
=3\sqrt{3}(\ell\tau+1)^{\frac12}\zeta_{24}^\ell e\left(\mp\frac16\pm \frac{v+m }{3} - \frac{v\ell }{6} + \frac{v^2\ell}{6}+\frac{vm \ell}{3}\right)
\int_{\frac{1}{\ell}}^{i\infty} \frac{g_{\frac{v\ell}{3}\pm \frac13 + \frac12,-v-m  + \frac12}(3\rho)}{\sqrt{-i(\rho + \tau)}} d\rho. 
\end{multline}  
Differentiating \eqref{eq:TvPeriod} yields
\begin{align*}
3(-1)^m  &e\left(-\frac{S_\ell\tau}{6}\right)\left.\frac{d}{dv}\left[T_1(v;\tau) + T_2(v;\tau)\right]\right|_{v=r}\\ 
&=2\pi i \sqrt{3}(\ell\tau+1)^{\frac12}\zeta_{24}^\ell \left(\pm 1 - \frac{\ell}{2} + 3\frac{\alpha}{\beta}\ell \right )e\left( \mp\frac{1}{6} \pm \frac{\alpha}{\beta}\right) \int_{\frac{1}{\ell}}^{i\infty} \frac{g_{\pm \frac13 + \frac12,-3\frac{\alpha}{\beta} + \frac12}(3\rho)}{\sqrt{-i(\rho + \tau)}} d\rho 
\\ 
 &\hspace{.5in}+ \sqrt{3}(\ell\tau+1)^{\frac12}\zeta_{24}^\ell e\left(\mp\frac16\pm \frac{\alpha}{\beta}\right)\left.\frac{d}{du}  \int_{\frac{1}{\ell}}^{i\infty} \frac{g_{\ell u \pm \frac13 + \frac12,-3u + \frac12}(3\rho)}{\sqrt{-i(\rho + \tau)}} d\rho\ \right|_{u=\frac{\alpha}{\beta}}.
\end{align*}
Here we use Proposition \ref{prop_Zg} (1) and the chain rule with $u=\frac{v+m }{3}$.

All together, we finally have 
\begin{align} \notag
&\mathcal H_{n,N}^-(S_\ell\tau) - \zeta_{24}^\ell (\ell\tau+1)^{\frac32} \mathcal H_{n,N}^-(\tau)& \\ \notag
&=\sum_{j=1}^N \frac{\zeta_{2\beta_j}^{-3\alpha_j} - \zeta_{2\beta_j}^{-5\alpha_j}}{\Pi_j(0)} \\ \notag
&\hspace{.2in}\times\Bigg[\sqrt{3}(\ell\tau+1)^{\frac12}\zeta_{24}^\ell\left(\frac{\ell}{4}-\frac{3\alpha_j}{2\beta_j}\ell\right)e\left(2\frac{\alpha_j}{\beta_j}\right)\sum_{\pm}\mp e\left(\mp \frac16\right) \int_{\frac{1}{\ell}}^{i\infty} \frac{g_{\pm\frac13+\frac12,-3\frac{\alpha_j}{\beta_j}+\frac12}(3\rho)}{\sqrt{-i(\rho+\tau)}} d\rho\\ 
\label{eqn_derivDtil} 
&\hspace{.8in}-\frac{\sqrt{3}(\ell\tau+1)^{\frac12}\zeta_{24}^\ell}{4\pi i} e\left(2\frac{\alpha_j}{\beta_j}\right)\sum_{\pm}\mp e\left(\mp \frac16\right) \left.\frac{d}{du} \int_{\frac{1}{\ell}}^{i\infty} \frac{g_{\ell u \pm \frac13 + \frac12,-3u + \frac12}(3\rho)}{\sqrt{-i(\rho + \tau)}} d\rho\ \right|_{u=\frac{\alpha_j}{\beta_j}}\\ \notag
&\hspace{2.5in}+\frac12(\ell\tau+1)^{\frac32}\zeta_{24}^\ell q^{-\frac{1}{24}}e\left(\frac32\frac{\alpha_j}{\beta_j}\right) -\frac12 e\left(-\frac{S_\ell\tau}{24}\right)e\left(\frac32\frac{\alpha_j}{\beta_j}\right)\Bigg]\\ 
\notag
&=\sum_{j=1}^N \frac{\zeta_{2\beta_j}^{\alpha_j} - \zeta_{2\beta_j}^{-\alpha_j}}{2\Pi_j(0)}  \Bigg[(\ell\tau+1)^{\frac12}\zeta_{24}^\ell\left(\frac{\ell}{2}-3\frac{\alpha_j}{\beta_j}\ell\right) H_{\alpha_j,\beta_j}(\tau)-\frac{(\ell\tau+1)^{\frac12}\zeta_{24}^\ell}{2\pi i}D_{\alpha_j,\beta_j}(\tau)
\\ \notag
&\hspace{3.4in}
+\zeta_{2\beta_j}^{-\alpha_j}\left((\ell\tau+1)^{\frac32}\zeta_{24}^\ell q^{-\frac{1}{24}} -e\left(-\frac{S_\ell\tau}{24}\right)\right) \Bigg],
\end{align} 
where we justify bringing the derivative inside the integral defining $D_{\alpha,\beta}$ in the proof of Proposition \ref{prop_analyH} below. 
\end{proof} 
To finish the proof of Theorem \ref{thm_main_NN}, it remains to show the following.
 
\begin{proposition}\label{prop_analyH}
The function  $\mathcal H_{n,N}^-(S_\ell\tau) - (\ell\tau+1)^{\frac32} \zeta_{24}^\ell\mathcal H_{n,N}^-(\tau)$ is analytic on $\R-\{\frac{-1}{\ell}\}$.
\end{proposition}
\begin{proof}   We use Proposition \ref{prop_hmintsf}. 
Clearly, the function $\mathcal E_2$ is analytic on $\R-\{\frac{-1}{\ell}\}$.    Moreover, \cite[Proposition 4.1]{WIN}  establishes the same for the function $H_{\alpha,\beta}$.   
Thus, it suffices to show that 
 the function $D_{\alpha,\beta}(\tau)$  is analytic in $\mathbb R \setminus \{-\frac{1}{\ell}\}$.  
We begin by computing
\begin{align*} \frac{d}{du} 
&g_{\ell u \pm \frac13 + \frac12,-3u + \frac12}(3\rho) \\& = \ell \sum_{n\in\mathbb Z} e({\tfrac32 \rho(n+\ell u \pm \tfrac13 + \tfrac12)^2}) e((n+\ell u \pm \tfrac13+\tfrac12)(-3u + \tfrac12)) \\ & + 2\pi i \ell (\tfrac12-3u) \sum_{n\in \mathbb Z}(n + \ell u \pm \tfrac13 + \tfrac12) e({\tfrac32 \rho(n+\ell u \pm \tfrac13 + \tfrac12)^2}) e((n+\ell u \pm \tfrac13+\tfrac12)(-3u + \tfrac12))
\\&+ 6\pi i (\ell\rho-1)\sum_{n \in \mathbb Z} (n + \ell u \pm \tfrac13 + \tfrac12)^2 e({\tfrac32 \rho(n+\ell u \pm \tfrac13 + \tfrac12)^2}) e((n+\ell u \pm \tfrac13+\tfrac12)(-3u + \tfrac12)).
\end{align*}
Since we will take the derivative in $u$ at the point $u=\frac{\alpha}{\beta}$, it  suffices to assume $|u-\frac{\alpha}{\beta}|<\epsilon$ for some sufficiently small $\epsilon>0$ as before.   Hence, we have that 
 $\ell (u-\frac{\alpha}{\beta}) \pm \tfrac13  \in (-\frac12,\frac12)$, so that 
$$\frac{\partial }{\partial u} g_{\ell u \pm \frac13 + \frac12,-3u + \frac12}(3\rho) \ll_u |\rho|
e^{-3\pi \operatorname{Im}(\rho) \big(N+\ell u \pm \tfrac13 + \tfrac12\big)^2}
$$ for some fixed $N \in \mathbb Z$.    Thus, we may apply the Leibniz Rule for indefinite integrals to the derivative  (in \eqref{eqn_derivDtil}) \begin{align}\label{def_Dabtil} \sqrt{3} \sum_{\pm}  e\left(\mp\frac{1}{6}\right)  \left.\frac{d}{du} \int_{\frac{1}{\ell}}^{i\infty} \frac{g_{\ell u \pm \frac13 + \frac12,-3u + \frac12}(3\rho)}{\sqrt{-i(\rho +\tau)}}d\rho\ \right|_{u=\frac{\alpha}{\beta}},\end{align}  
 and that deduce $D_{\alpha,\beta}$ is analytic for $\tau \in \mathbb R \setminus \{-\frac{1}{\ell}\}$.   
\end{proof}

 \ \\


\begin{thebibliography}{BOR}

\bibitem{Andrews} G. E. Andrews {\it Partitions, Durfee symbols, and the Atkin-Garvan moments of ranks}, Invent. Math.  169  (2007), 37--73.

%\bibitem{AtkinSD} A. O. L. Atkin and H. P. F. Swinnerton-Dyer, \emph{Some properties of partitions},
Proc. London Math. Soc.  66  (1954), 84--106.

\bibitem{Bri1} K. Bringmann, \emph{On the explicit construction of higher deformations of partition statistics}, Duke Math. J., \text{144} (2008), 195-233.

\bibitem{BFOR} K. Bringmann, A. Folsom, K. Ono, and L. Rolen, \emph{Harmonic Maass forms and mock modular forms: theory and applications,}  American Mathematical Society Colloquium Publications 64. American 
                                        Mathematical Society, Providence, RI, 2018.  391pp.

\bibitem{BGM} K. Bringmann, F. Garvan, and K. Mahlburg, \emph{Partition statistics and quasiweak Maass forms}, Int. Math. Res. Notices,  \text{1} (2009), 63--97.

\bibitem{BO} K. Bringmann and K. Ono, \emph{Dyson's ranks and Maass forms}, Ann. of Math., \text{171} (2010), 419-449.

\bibitem{BR} K. Bringmann and L. Rolen, \emph{Radial limits of mock theta functions,}
Res. Math. Sci. 2 (2015), Art. 17, 18 pp. 



\bibitem{BF} J. Bruiner and J. Funke, \emph{On two geometric theta lifts,} Duke Math. J. {{125}} (2004), 45-90.

\bibitem{CLR} D. Choi, S. Lim, and R. Rhoades, {\it Mock modular forms and quantum modular forms}, Proc. Amer. Math. Soc. 144  (2016), no.6 2337--2349.


\bibitem{Dyson} F. Dyson, \emph{Some guesses in the theory of partitions},
Eureka (Cambridge) 8 (1944), 10--15. 

\bibitem{FGKST} A. Folsom, S. Garthwaite, S-Y Kang, H. Swisher, and S. Treneer, \emph{Quantum mock modular forms arising from eta-theta functions,} Res. Number Theory 2 (2016), Art. 14, 41 pp.

\bibitem{FKTY} A. Folsom, C. Ki, Y.N. Truong Vu, and B. Yang, \emph{``Strange'' combinatorial quantum modular forms,} J. Number Theory 170 (2017), 315--346.


\bibitem{F-K} A. Folsom and S. Kimport {\it Mock modular forms and singular combinatorial series}, Acta Arith. 159 (2013), 257--297.


\bibitem{WIN} A. Folsom, S. Kimport, M-J Jang, and H. Swisher, {\it Quantum modular forms and singular combinatorial series with distinct roots of unity}, Springer Research Directions in Number Theory: Women in Numbers IV, accepted for publication. 

\bibitem{FOR} A. Folsom, K. Ono, and R.C. Rhoades, \emph{Mock theta functions and quantum modular forms,}
Forum Math. Pi 1 (2013), e2, 27 pp. 

\bibitem{Knopp}
M.~I. Knopp.
\newblock {\em Modular functions in analytic number theory}.
\newblock Markham Publishing Co., Chicago, Ill., 1970.

\bibitem{OnoCDM} K. Ono, \emph{Unearthing the visions of a master: harmonic Maass forms and number theory,} Current developments in mathematics, (2008), 347-454, Int. Press, Somerville, MA, (2009).

\bibitem{Rad} H. Rademacher, \emph{Topics in analytic number theory}, Die Grundlehren der math. Wiss., Band 169, Springer-Verlag, Berlin, (1973).


\bibitem{ZagierB} D. Zagier, \emph{Ramanujan's mock theta functions and their applications (after Zwegers and Ono-Bringmann),} 
S\'eminaire Bourbaki Vol. 2007/2008, 
Ast\'erisque {{326}} (2009), Exp. No. 986, vii-viii, 143-164 (2010). 

\bibitem{Zlec} D. Zagier, International conference:  Mock theta functions and applications in combinatorics, algebraic geometry and mathematical physics, Max Planck Institute for Mathematics, Bonn, Germany, (2009).

\bibitem{Zqmf} D.Zagier,  \emph{Quantum modular forms}, Quanta of maths, 659-675, Clay Math. Proc., 11, Amer. Math. Soc., Providence, RI, 2010.

\bibitem{Zwegers1} S. Zwegers, {\it Mock theta functions}, Ph.D. Thesis, Universiteit Utrecht, 2002.

\bibitem{Zwegers2} S. Zwegers, {\it Multivariable Appell functions}, Preprint.

\end{thebibliography}
\end{document}